\documentclass[12pt,reqno,a4paper]{amsart}

\usepackage{float}
\usepackage[bf,hypcap]{caption}
\usepackage{a4}
\usepackage{amsmath}
\usepackage{amsmath}
 \usepackage{amsfonts}
 \usepackage{amsthm}
 \usepackage{amssymb}
 \usepackage{graphicx}
\usepackage{amsfonts}
\usepackage{amsthm}
\usepackage{graphicx}
\usepackage[colorlinks=true,citecolor=black,linkcolor=black,urlcolor=blue]{hyperref}
\def\ds{\displaystyle}

\makeatletter
\newtheorem{theorem}{Theorem}[section]

\newtheorem{lemma}[theorem]{Lemma}
\newtheorem{corollary}{Corollary}[theorem]

 \newtheorem{exm}[theorem]{Example}
 \newtheorem{remark}[theorem]{Remark}

\numberwithin{equation}{section}

\begin{document}
\title[{\tiny  On  a product of three theta functions and the number of repre...  }]{ On  a product of three theta functions and the number of representations of integers as  mixed ternary  sums involving squares, triangular, pentagonal and octagonal  numbers }
 \maketitle
 \begin{center}{ N. A. S. Bulkhali$^1$, G. Kavya Keerthana$^2$ and Ranganatha D.$^2$ \\
 \vspace{.5cm}
  
  $^1$Department of Mathematics, Faculty of Education-Zabid,\\ Hodeidah University, Hodeidah, Yemen\\
  $^2$Department of Mathematics, Central University of Karnataka\\Kalaburagi-585376, India\\
E-mail: nassbull@hotmail.com,  kavya.27598@gmail.com and ddranganatha@gmail.com }
 \end{center}


\begin{abstract}
 In this paper, we derive a general formula to express the product
of three  theta functions as a linear combination of other products of three
theta functions. Moreover, we use the main formula to deduce a general formula  for the
product of two theta functions. Furthermore,  as applications, we extract several  theorems
in the theory of representation of integers as  mixed ternary sums involving  squares, triangular numbers, generalized  pentagonal numbers and generalized octagonal  numbers.  \\

\noindent\textsc{2000 Mathematics Subject Classification.} 11D72, 11E20, 11E25, 11F27,  14H42.\\

 \noindent\textsc{Keywords and phrases.}
Ramanujan's theta function,  sums of squares, triangular numbers, ternary quadratic forms.
\end{abstract}
\section{Introduction}\label{Introduction}
Ramanujan's general theta function is defined by
\begin{equation}\label{2.1}f(a,b):= \sum_{n=-\infty}^{\infty}{{a}^{n(n+1)/2}}{{b}^{n(n-1)/2}},~~~|ab|<1.\end{equation}

Theta functions have applications in several areas, for instance, in the theory of representation of integers, the theory of  elliptic
functions,  analytic number theory, modular forms  and Riemann surfaces. Theta
functions are also important in physics, which appear in the partition function of strings and two-dimensional
conformal field theories and can appear in certain areas of statistical mechanics. Due to these and many other applications, theta function identities have  great importance, great benefits and are worthy to studies more and more.
The following general formula is due to S. Ramanujan~\cite[ Entry 31]{Adiga3}, in which
theta function is  expressed as a linear combination of other theta functions:\\
Let $U_n=a^{n(n+1)/2}b^{n(n-1)/2}$ and $V_n=a^{n(n-1)/2}b^{n(n+1)/2}$ for each integer $n$. Then
\begin{equation}\label{RamIden}
  f(U_1,V_1)=\sum_{r=0}^{n-1}U_r f \left( \frac{U_{n+r}}{U_r},\frac{V_{n-r}}{U_r}  \right).
\end{equation}

Many mathematicians have studied the product of two theta functions and established several general formulas.
For instance, L. J. Rogers
\cite{Rogers2} has found some general identities based on
Schr\"{o}ter-type theta function identities.  D. Bressoud~\cite{Bressoud1}, \cite{Bressoud2}, H. Yesilyurt~\cite{Yesilyurt2}  and  L.-C. Zhang~\cite{Zhang} have generalized   Rogers' approach.
S.-L. Chen and S.-S. Huang~\cite{SCHEN} have slightly generalized the approach of G. Watson~\cite{Watson2} and established some identities for the products of two  theta functions. C. Adiga and the author~\cite{Adiga2} have derived  general formulas to
express the product of two  theta functions as  linear combinations of other products of two  theta functions.
In~\cite{Bulkhali}, the author employed  the main theorem found in \cite{Adiga2} to establish  new identities that are
analogous to the famous  Schr\"{o}ter's formulas, found in Chapter 16~\cite{Adiga3}.
R. Blecksmith, J. Brillhart, and I. Gerst~\cite{Blecksmith} obtained a general formula
 for the product of
two theta functions as a sum of  products of pairs of theta
functions; this formula generalizes
formulas of H. Schr\"{o}ter which can be found in~\cite{Adiga3}.
Q. Yan~\cite{Yan} has used the technique of L. Carlitz and M. V.
Subbarao~\cite{Carlitz}, to establish a general formula  for the
 product of two Jacobi's triple products found  in
\cite{Chu1}. \\

A few works have been done for the product of three theta functions.
L.-C.~Shen~\cite{Shen} employed the techniques of the Fourier series to derive identities for products of three classical theta functions. As a consequence of his results,  Shen obtained a quintuple product identity and two identities of Ramanujan; some of Shen's identities involve Borweins' cubic theta functions \cite{Borwein}.
Later,  P. C. Toh~\cite[Theorem 3.1]{Toh}, reproved some identities established by Shen in equivalent forms.
 Motivated by the work of Shen,  X. M. Yang \cite{Yang} studied the products of three classical theta functions and established two theta function identities with four parameters using the theory
of theta functions and introduced common generalizations of Hirschhorn--Garvan--Borwein cubic theta functions \cite{HirschhornGar.Borwein} and also re-derived the quintuple product identity, one of Ramanujan's
identities, Winquist's identity and many other interesting identities.
  Cao~\cite{Cao} deduced some new identities for the products of three  theta functions. More recently, H.-C. Zhai, J. Cao and S. Arjika~\cite{ZhaiCaoArjika} established several expansion formulas for products of the Jacobi theta
functions. As applications, they derived some expressions of the powers of $(q; q)_{\infty}=\prod_{n=0}^{\infty}(1-q^{n+1})$ where $|q|<1$, by using these
expansion formulas and some of these expressions are for the product of three theta functions.
Utilizing the correspondence between product identities
for Ramanujan's theta functions and integer matrix exact covering systems,  Z. Cao~\cite{Cao} gave a general theorem
to write a product of  theta functions as  linear combinations
of other products of theta functions,  which is a generalization for many theta function identities found in the literature.\\

Motivated from the above work, in Section \ref{sec.Main}, we prove  a general formula to write the product
of three  theta functions as a linear combination of other products of three
theta functions. To prove our main theorem, we utilize Watson's technique \cite{Watson2},
 which he has used to prove some of  Ramanujan's forty identities.  Here, we state our main theorem

\begin{theorem}\label{MainThm3}  Let $k$ and $r$ be positive integers with $k>r$, $\gcd (2k,r)=1$ or 2  and $\gcd(2k,k-r)=1$. For any integers $g$, $h$, $u$, $v$, $i$ and $j$ with
$g+h=S_1>0$, $g-h=D_1 $, $u+v=S_2>0$,  $u-v=D_2 $, $i+j=S_3>0$, $i-j=D_3$, $ S_1=S_2$ and $2\,S_{{1}}=r \left( k-r \right) S_{{3}}$, we have
{\small
\begin{align}\label{3.3}
f& \left( \epsilon _{{1}}{q}^{g},\epsilon _{{1}}{q}^{h} \right) f
 \left( \epsilon _{{2}}{q}^{u},\epsilon _{{2}}{q}^{v} \right) f
 \left( \epsilon _{{3}}{q}^{i},\epsilon _{{3}}{q}^{j} \right)
 \nonumber \\= &f_{{\delta_{{1}}+\delta_{{2}}}} \left( {q}^{(r \left( k-r \right)
S_{{3}}+D_{{1}}-D_{{2}})/2},{q}^{(r \left( k-r \right) S_{{
3}}-D_{{1}}+D_{{2}})/2} \right)
 \sum_{\alpha={\left[\frac{2 - k}{2}\right]}}^{\left[\frac{k}{2}\right]}(-1)^{\alpha \delta_3} q^{\alpha \left({\alpha}S_{{3}}+D_{{3}} \right)/2} \nonumber \\ & \times f_{{\delta_{{1}}+\delta_{{2}}+r\delta_{{3}}}} \left( {q}^{\left(r
 \left( S_{{3}} \left( k+2 \, \alpha \right) +D_{{3}} \right) +D_{{1}}
+D_{{2}} \right)/2},{q}^{ \left(r \left( S_{{3}} \left( k- 2\, \alpha \right) -D_{
{3}} \right) -D_{{1}}-D_{{2}}\right)/2} \right) \nonumber \\ & \times f_{{\delta_{{1}}+\delta_{{2}}+ \delta_{{3}}}}
 \left( {q}^{\left( \left( k-r \right)  \left( S_{{3}} \left( k-2\,
\alpha \right) -D_{{3}} \right) +D_{{1}}+D_{{2}}\right)/2},{q}^{\left(
 \left( k-r \right)  \left( S_{{3}} \left( k+2\,\alpha \right) +D_{{3}
} \right) -D_{{1}}-D_{{2}}\right)/2} \right) \nonumber \\ & + \left( -1 \right) ^{\delta_{{3}} \left( k-r+1 \right)/2 } f_{{\delta_{{1}}+\delta_{{2}}}} \left( {q}^{ \left(2\,r \left( k-r \right) S_
{{3}}+D_{{1}}-D_{{2}}\right)/2},{q}^{\left(-D_{{1}}+D_{{2}}\right)/2} \right)
\nonumber \\ & \times \Bigg\{ \sum_{\alpha=1}^{\left[\frac{k+1}{2}\right]} (-1)^{\alpha \delta_3+ {\delta_1}} q^{ \left(S_{{1}}+D_{{1}}+\frac{1}{4}\,S_{{3}} \left( -k+r+2\,\alpha-1 \right) ^{2}+\frac{1}{2}
\,D_{{3}} \left( -k+r+2\,\alpha-1 \right)
 \right)/2 }  \nonumber \\ & \times  f_{{\delta_{{1}}+\delta_{{2}}+r\delta_{{3}}}} \left( {q}^{\left(r
 \left( S_{{3}} \left( k+2\,\alpha-1 \right) +D_{{3}} \right) +D_
{{1}}+D_{{2}}\right)/2},{q}^{\left(r \left( S_{{3}} \left( k-2\,\alpha+1
 \right) -D_{{3}} \right) -D_{{1}}-D_{{2}}\right)/2} \right)
 \nonumber \\ & \times  f_{{\delta_{{1}}+\delta_{{2}}+  \delta_{{3}}}}
 \left( {q}^{\left( \left( k-r \right) \left( S_{{3}} \left( 2\,k-2\,\alpha+1 \right) -D_{
{3}} \right)   +D_{{1}}+D_{{2}}\right)/2},{q}^{\left( \left( k-r \right) \left( S_{{3}} \left(2\, \alpha-1 \right) +D_{{3}} \right)   -D_{{1}}-D_{{2}}\right)/2} \right) \nonumber \\ & +
\sum_{\alpha=1}^{\left[\frac{k}{2}\right]} (-1)^{\alpha \delta_3+ {\delta_2}} q^{ \left(S_{{2}}-D_{{2}}+\frac{1}{4}\,S_{{3}} \left( k-r-2\,\alpha+1 \right) ^{2}+\frac{1}{2}
\,D_{{3}} \left( k-r-2\,\alpha+1 \right)
 \right)/2 }
  \nonumber \\ & \times  f_{{\delta_{{1}}+\delta_{{2}}+r\delta_{{3}}}} \left( {q}^{(r
 \left( S_{{3}} \left( k-2\,\alpha+1 \right) +D_{{3}} \right) +D_
{{1}}+D_{{2}})/2},{q}^{(r \left( S_{{3}} \left( k+2\,\alpha-1
 \right) -D_{{3}} \right) -D_{{1}}-D_{{2}})/2} \right)  \nonumber \\ & \times  f_{{\delta_{{1}}+\delta_{{2}}+  \delta_{{3}}}}
 \left( {q}^{( \left( k-r \right)  \left( S_{{3}} \left( 2\,\alpha
-1 \right) -D_{{3}} \right) +D_{{1}}+D_{{2}})/2},{q}^{(
 \left( k-r \right)  \left( S_{{3}} \left( 2\,k-2\,\alpha+1 \right) +D
_{{3}} \right) -D_{{1}}-D_{{2}})/2} \right)
   \Bigg\} .
\end{align}}
Here $\epsilon_i \in \{-1,1\}$,   $\delta_i=
\ds\frac{1-\epsilon_i}{2}$ for $i=1,2,3$ and $[x]$ denote the
greatest integer less than or equal to $x$.
\end{theorem}
Theorem \ref{MainThm3} can be considered as a generalization for the product of two theta functions.  In Section \ref{pro2thetaF}, using this theorem, we establish some identities for the product of two theta functions that are analogous to some identities found in the literature.

Theorem \ref{MainThm3} has many applications; here, we use it to obtain several theorems in the theory of representations of integers.
Let $\mathbb{N}$ denote the set of all positive integers, let $\mathbb{N}_0=\mathbb{N} \cup \left\{ 0\right\} =\{0,1,2,3,\ldots\}$ and  $\mathbb{Z}$ be the set of all integers. Let  $t_n$  denote the $n^{th}$ triangular number  $t_n = n(n + 1)/2$ with $n \in \mathbb{N}_0$. Let  $p_n$  denote the $n^{th}$ generalized  pentagonal  number  $p_n = n(3\,n + 1)/2$ with $n \in \mathbb{Z}$. Let  $g_n$  denote the $n^{th}$  generalized octagonal  number  $g_n = n(3\,n + 2)$ with $n \in \mathbb{Z}$. It is understood that the squares are $n^2$ with $ n \in \mathbb{Z}$. We define for $N \in \mathbb{N}_0$ ,
\begin{align*}
  r(a_1,a_2,a_3;N):=& \left|  \left\{ (l,m,n) \in \mathbb{Z}^3 | N= a_1\,l^2 +a_2 \,m^2+a_3n^2  \right\}  \right|,\\
   T(a_1,a_2,a_3;N):=& \left|  \left\{ (l,m,n) \in \mathbb{N}_0^3| N= a_1\,t_{l} +a_2 \,t_{m}+a_3t_{n}  \right\}  \right|,
\end{align*}
where $a_1, a_2,  a_3 \in \mathbb{N}$.\\
The theory of the representations of integers has been studied by many great mathematicians such as  Liouville,   Euler, Legendre,  Dirichlet,  Gauss, Eisenstein and Ramanujan (for example see \cite{Dickson2}, \cite{Dickson} and \cite{WhittakerWatson}).
 Hirschhorn \cite{Hirschhorn} established many representation theorems
involving  squares, triangular numbers, pentagonal numbers and octagonal
numbers. A similar work have been done by N. D. Baruah and B. K. Sarmah \cite{BaruahSarmah}.
Recently, Sun in \cite{Sun4}  derived  many relations between $r(a,b,c;8N+a+b+c)$ and $T(a,b,c;N)$ and conjectured some relations proved later by many mathematicians. Motivated by this, in Section \ref{Applications}, as applications of Theorem \ref{MainThm3}, we establish several relations between the number of representations of integers as  ternary sums involving squares, triangular numbers, generalized pentagonal numbers and generalized octagonal numbers.

In $1796$, C. F. Gauss proved his famous result that every integer can be written
as the sum of three triangular numbers. In $1862$, J. Liouville (cf. \cite[p. 23]{Dickson}) proved the following result:\\
 Let $a, b, c$ be positive integers with $a \leq b \leq c$. Then every
$n \in \mathbb{N}_0$ can be written as $a\,t_x + b\, t_y + c\, t_z$ with $x, y, z \in \mathbb{Z}$, if and only if $(a, b, c)$ is
among the following vectors:
\begin{align*}(1, 1, 1), (1, 1, 2), (1, 1, 4), (1, 1, 5), (1, 2, 2), (1, 2, 3), (1, 2, 4).\end{align*}
In $2007$, Z.-W. Sun \cite{Sun} proved the following two results:
\begin{enumerate}
  \item Let $a, b, c$ be positive integers with $a \leq b$. Suppose that every
$n \in  \mathbb{N}_0$ can be written as $a\,x^2 + b\,y^2 + c\,t_z$ with $x,y,z \in \mathbb{Z}$, then $(a, b, c)$ is among the following
vectors:
\begin{align*}(1, 1, 1), (1, 1, 2), (1, 2, 1), (1, 2, 2), (1, 2, 4),\\
(1, 3, 1), (1, 4, 1), (1, 4, 2), (1, 8, 1), (2, 2, 1).\end{align*}
  \item Let $a, b, c$ be positive integers with $b \geq c$. Suppose that every
$n \in \mathbb{N}_0$ can be written as $a\,x^2 + b\,t_y + c\,t_z$ with $x, y, z \in \mathbb{Z}$. Then $(a, b, c)$ is
among the following vectors:
\begin{align*}&(1, 1, 1), (1, 2, 1), (1, 2, 2), (1, 3, 1), (1, 4, 1), (1, 4, 2), (1, 5, 2),\\
(1, 6&, 1), (1, 8, 1), (2, 1, 1), (2, 2, 1), (2, 4, 1), (3, 2, 1), (4, 1, 1), (4, 2, 1).\end{align*}
\end{enumerate}
Observe that, the above results are true  when $x,y,z \in \mathbb{N}_0$ because $t_n=t_{-n-1}$ and $n^2=(-n)^2$ for any $n \in \mathbb{N}_0$.\\


The well-known theorem of Gauss and Legendre states that a non-negative $N$ is a sum of three squares if and only if it is not of the form $4^k(8l + 7)$ with $k,l \in \mathbb{ N}_0$.\\
In $1916$, Ramanujan conjectured that the only positive even numbers not of the form $x^2 + y^2 + 10z^2$ are those $4^k(16l + 6) ~ (k,l \in \mathbb {N})$, which was proved in 1927 by  Dickson~\cite{Dickson1}.
The form $ x^2 +2y^2 +6z^2$ represents all positive integers not of the form $4^k(8l + 5)  $ for some $ k,l \in \mathbb {N}$ (see \cite[pp. 111--112]{Dickson2}). Dickson~\cite[pp. 111--112]{Dickson2} listed 102 similar results.
Motivated by this work,   in Section \ref{Applications}, as consequences of our results, we deduce, for $l,m,n \in \mathbb{Z}$,  the following results:\\
Every non-negative integer of the form $4k+3$, ($ k \in \mathbb{N}_0$) cannot be represented in any of the following forms:
$$  l^2 +m^2+4t_n,  ~~~~ 4p_l+g_m+g_n,~~~~ 12t_l+g_m+g_n, ~~~~ 4t_l+g_m+g_n.$$
Every non-negative integer of the form $4k+1$, ($ k \in \mathbb{N}_0$) cannot be represented in any of the following forms:
$$ 3l^2+3m^2+4t_n, ~~~~~~~~3l^2+3m^2+4p_n.  $$
Every non-negative integer of the form $4k+2$, ($ k \in \mathbb{N}_0$) cannot be represented in any of the following forms:
$$  3l^2 +12t_m+g_n,  ~~~~ 3l^2+4p_m+g_n,~~~~ 9l^2+4p_m+3g_n, ~~~~ 3l^2+4t_m+g_n.$$

\section{Preliminary results}\label{Preliminary}
We need the following definitions and preliminary results:\\
The Jacobi triple product identity~\cite[ Entry 19]{Adiga3} in
Ramanujan's notation is
\begin{equation}f(a,b)= (-a;ab)_\infty(-b;ab)_\infty(ab;ab)_\infty.\nonumber\end{equation}
Using (\ref{2.1}),  we define
\begin{equation}\label{2.1a}
f_{\delta}(a,b) =
\begin{cases}
f(a,b) & \text{if } \delta \equiv 0 \pmod{2},
\\
f(-a,-b) & \text{if } \delta \equiv 1 \pmod{2}.
\end{cases}
\end{equation}
The function $f(a,b)$ satisfies the following
basic properties \cite{Adiga3}:
\begin{align}f(a,b)& = f(b,a)\nonumber\\
\label{2.5}f(-1,a)& = 0,
\intertext{and}
\label{2.4}f(1,a)& = 2f(a,a^{3}).
\end{align}
Ramanujan has defined the following  three special cases of
(\ref{2.1}) \cite[Entry 22]{Adiga3}:
\begin{align}\varphi(q)&:= f(q,q) =  \sum_{n=-\infty}^{\infty} q^{n^2},\nonumber\\
\label{2.8}\psi(q)&:= f(q,q^3) =  \sum_{n=-\infty}^{\infty} {q^{n(2n+1)}}=  \sum_{n=0}^{\infty} {q^{n(n+1)/2}},\\
\intertext{and}
f(-q):&= f(-q,-q^2) =  \sum_{n=-\infty}^{\infty} {(-1)^n}{q^{n(3n+1)/2}}.\nonumber
\end{align}
Also, after Ramanujan, we define
\begin{align*}
 X(q):&= f(q,q^2) =  \sum_{x=-\infty}^{\infty} {q^{x(3x+1)/2}},\\
Y(q):&= f(q,q^5) =  \sum_{x=-\infty}^{\infty} {q^{x(3x+2)}}.
\end{align*}
Using \eqref{2.8} and \eqref{2.4}, we observe that
\begin{equation}
  \sum_{n=-\infty}^{\infty} {q^{n(n+1)/2}}=2\, \sum_{n=0}^{\infty} {q^{n(n+1)/2}}.\nonumber
\end{equation}
In the rest of the proof of our main theorem, we need the definition of the sign function, for any nonzero real number $x$, which is given by
\begin{equation*}
\text{sign}\, x :=
\begin{cases}
1 & \text{if } x>0,
\\
-1 & \text{if } x<0.
\end{cases}
\end{equation*}

\begin{lemma}\label{lemma1}\cite[Entry 29(i)]{Adiga3}.
 If $ab=cd$, then
\begin{align*}
 f(a,b)f(c,d)+f(-a,-b)f(-c,-d)&=2f(ac,bd)f(ad,bc).
\end{align*}
\end{lemma}
We use  the following lemma very often in this paper:
\begin{lemma}\label{l3}  Let
 $m=\left[ {s}/({s-r}) \right],~ l= m
(s-r)-r,~ k= -m(s-r)+s$ and $h=mr-{m(m-1)(s-r)}/{2}$,~ $0\leq
r<s$. Here $[x]$ denote the largest integer less than or equal to
$x$. Then,
\begin{align*}
    f(q^{-r},q^s)&= q^{-h}f(q^l,q^k),\\
 f(-q^{-r},-q^s) &= (-1)^m q^{-h}f(-q^l,-q^k).
 \end{align*}
 \end{lemma}
For a proof of Lemma \ref{l3}, see~\cite[Lemma 1]{Adiga}. \\
\begin{lemma} \cite[Entry 30 (i)]{Adiga3}
	We have
	\begin{align}  \label{2.10P}
	f(a, ab^2) f(b, a^2b) =& f(a, b) \psi(ab).
\end{align}
\end{lemma}
The following identity is an easy consequence of \eqref{RamIden} when $n=2$:
\begin{equation}\label{P2.10}
    f(a,b)=f(a^3b,ab^3) + a f(b/a,a^5b^3).
\end{equation}
Using \eqref{P2.10}, we obtain
\begin{align}\label{varphi=}
  \varphi(q)&=\varphi \left(q^4 \right)+2\,q\,\psi \left(q^8\right),\\
  \psi(q)&=f\left(q^{10},q^{6}\right)+qf\left(q^{14},q^2\right),\nonumber\\
  X(q)&=f\left(q^7,q^5\right)+qf\left(q^{11},q\right).\label{X(q)=}
\end{align}

\section{Proof of Theorem \ref{MainThm3} and some corollaries}\label{sec.Main}

In this section, we prove Theorem \ref{MainThm3}, in which we represent the product of three theta functions as a linear combination of other products of three theta functions  and we deduce some corollaries.

\begin{proof}[Proof of Theorem \ref{MainThm3}]  From the statement of the theorem, we have two cases.\\

\textit{First Case}. When $\gcd (2k,r)=1 $ and $\gcd(2k,k-r)=1$.\\
Using (\ref{2.1}), we have
\begin{align}\label{3.a1}
  f \left( \epsilon _{{1}}{q}^{g},\epsilon _{{1}}{q}^{h} \right) &f
 \left( \epsilon _{{2}}{q}^{u},\epsilon _{{2}}{q}^{v} \right) f
 \left( \epsilon _{{3}}{q}^{i},\epsilon _{{3}}{q}^{j} \right)\nonumber \\ & =\sum _{
l,m,n=-\infty }^{\infty }{\epsilon _{{1}}}^{l}{\epsilon _{{2}}}^{m}{
\epsilon _{{3}}}^{n}{q}^{(S_{{1}}{l}^{2}+S_{{2}}{m}^{2}+S_{{3}}{n}^{2}+D_{{1}}l+D_{{2}}m+D_{{3}}n
)/2}.
\end{align}
 For convenience, we define
\begin{equation}\label{3.1aa}
 F(l,m,n):= {\epsilon _{{1}}}^{l}{\epsilon _{{2}}}^{m}{
\epsilon _{{3}}}^{n}{q}^{(S_{{1}}{l}^{2}+S_{{2}}{m}^{2}+S_{{3}}{n}^{2}+D_{{1}}l+D_{{2}}m+D_{{3}}n
)/2}.
\end{equation}
Using \eqref{3.1aa}, we may rewrite \eqref{3.a1} in the form
\begin{equation}\label{3.1aa2}
  f \left( \epsilon _{{1}}{q}^{g},\epsilon _{{1}}{q}^{h} \right) f
 \left( \epsilon _{{2}}{q}^{u},\epsilon _{{2}}{q}^{v} \right) f
 \left( \epsilon _{{3}}{q}^{i},\epsilon _{{3}}{q}^{j} \right) =\sum _{
l,m,n=-\infty }^{\infty }F(l,m,n).
\end{equation}
In this representation, we make the change of indices by setting
\begin{equation}\label{3.a2}
  \left( k-r \right) l+ \left( k-r \right) m+2\,n=2k\,L+a , ~~~  l-m=2\,M+b,~~~
  rl+rm-2\,n=2k\,N+c,
\end{equation}
where $a$ and $c$ have values selected from the complete set of
residues modulo $2\,k$. Here, we take $a\in \{0, \pm1,\pm2,\ldots,\pm(k-1),k \}$ and $c\in \{0, \pm1,\pm2,\ldots,\pm(k-1),-k \}$. The values of $b$ are  selected from the set $\{0,1\}$.
Now, solving the simultaneous
system \eqref{3.a2}, we have

\begin{align} \label{3.a3}
 l=&L+M+N+{\frac {a+kb+c}{2k}}, \nonumber \\
 m=&L-M+N+{\frac {a-kb+c}{2k}},\\
 n=&rL- \left( k-r \right) N+{\frac {ra- \left( k-r \right) c}{2k}}.\nonumber
\end{align}
Since $l,m,n,L,M$ and $N$ are all integers, we see that values of $a$, $b$ and $c$ are associated as follows:
\begin{equation}\label{3.a4}
\left\{
  \begin{array}{ll}
   b=0 ~~~\text{and}~~~ c=-a , & \hbox{if $a$ is  even;} \\
    b=1~~~ \text{and}~~~ c= (\text{sign} \, a)\,k-a, & \hbox{if $a$ is  odd.}
  \end{array}
\right.\end{equation}
Since $\gcd (2k,r)=1 $ and $\gcd(2k,k-r)=1$, so we must have $\gcd (k-r,r)=1$.
Thus, there is one-to-one correspondence between the set of all triples of integers $(l, m, n)$, $-\infty <l,m,n<\infty,$ and all
sets of integers $(L,M,N, a)$, $-\infty< L, M, N < \infty$, $-(k-1)\leq a\leq k$.\\
Now, from \eqref{3.a4} we have  the following  cases:\\

\textit{Case} $1.$ If $a$ is an even number, we put $a=2 \alpha$,   $b=0$  and $c=- 2 \alpha$ in the system  \eqref{3.a3} to obtain
\begin{equation} \label{3.a5}
 l=L+M+N,  \qquad
 m=L-M+N,\qquad
 n=rL- \left( k-r \right) N+\alpha.
\end{equation}

\textit{Case} $2.$ If $a$ is an odd number, we put $a=2 \alpha - 1$,   $b=1$  and $c= (\text{sign} \, a)\,k- 2\alpha + 1$ in the system  \eqref{3.a3} to obtain
\begin{align} \label{3.a6}
 l=&L+M+N+{\frac {1+(\text{sign} \, a)}{2}}, \nonumber \\
 m=&L-M+N+{\frac {-1+(\text{sign} \, a)}{2}},\\
 n=&rL- \left( k-r \right) N+{\frac {-(\text{sign} \, a)k+(\text{sign} \, a)r+2 \alpha -1 }{2}}.\nonumber
\end{align}
 One can easily observe that $k$ is  even  and $r$ is  odd, so that the quotient  $(-(\text{sign} \, a)k+(\text{sign} \, a)r+2 \alpha -1 )/2$ is integer. This case can be separated into  two parts  according to sign of $a$. From \eqref{3.1aa2},  \eqref{3.a5} and \eqref{3.a6}, we deduce
{\small
\begin{align}\label{3.a7}
   f& \left( \epsilon _{{1}}{q}^{g},\epsilon _{{1}}{q}^{h} \right) f
 \left( \epsilon _{{2}}{q}^{u},\epsilon _{{2}}{q}^{v} \right) f
 \left( \epsilon _{{3}}{q}^{i},\epsilon _{{3}}{q}^{j} \right) \nonumber \\ =& \sum_{\alpha=-{\frac{k-2}{2}}}^{\frac{k}{2}} \sum _{
L,M,N=-\infty }^{\infty } F \left(L+M+N,  L-M+N, rL- \left( k-r \right) N+\alpha \right)
\nonumber \\ & + \sum_{\alpha=1}^{\frac{k}{2}} \sum _{L,M,N=-\infty }^{\infty } F\left(L+M+N+1,  L-M+N, rL- \left( k-r \right) N+\frac{-k+r+2\,\alpha-1}{2} \right)
\nonumber \\ &+  \sum_{\alpha=1}^{\frac{k}{2}} \sum _{L,M,N=-\infty }^{\infty } F\left(L+M+N,  L-M+N-1, rL- \left( k-r \right) N+\frac{k-r-2\,\alpha+1}{2} \right),
\end{align}}
where we changed $a$ to $-a$ (or $2\alpha-1$ to $-2\alpha+1$) in the summand of the third summation above. We observe that the number of terms in the right-hand side of \eqref{3.a7} are $2k$.\\
Since  $2S_1=r(k-r)S_3$ and $\epsilon_i=(-1)^{\frac{1-\epsilon_i}{2}},$ for
$i=1,2,3$, the above identity  can be written as

\begin{align*}
  f& \left( \epsilon _{{1}}{q}^{g},\epsilon _{{1}}{q}^{h} \right) f
 \left( \epsilon _{{2}}{q}^{u},\epsilon _{{2}}{q}^{v} \right) f
 \left( \epsilon _{{3}}{q}^{i},\epsilon _{{3}}{q}^{j} \right) = \sum_{\alpha=-{\frac{k-2}{2}}}^{\frac{k}{2}}\left( -1 \right) ^{\alpha\,\delta_{{3}}}{q}^{({\alpha}^{2}S_{{3}}+
\alpha\,D_{{3}})/2}  \end{align*}
 \begin{align*} \times & \sum _{L=-\infty }^{\infty } \left( -1 \right) ^{ \left(
\delta_{{1}}+\delta_{{2}}+r\delta_{{3}} \right) L}{q}^{ \left(krS_{{3}}{L}^{2}+ \left( 2\,
r\alpha\,S_{{3}}+rD_{{3}}+D_{{1}}+D_{{2}} \right) L \right)/2}
 \\ \times & \sum _{M=-\infty }^{\infty } \left( -1 \right) ^{ \left( \delta_{{1}}-
\delta_{{2}} \right) M}{q}^{\left(r \left( k-r \right) S_{{3}}{M}^{2}+
 \left( D_{{1}}-D_{{2}} \right) M \right)/2}\\\times &  \sum _{N=-\infty }^{\infty } \left( -1 \right) ^{ \left( \delta_{{1}}+
\delta_{{2}}- \left( k-r \right) \delta_{{3}} \right) N}{q}^{ \left(k \left(
k-r \right) S_{{3}}{N}^{2}+ \left( D_{{1}}+D_{{2}}-2\,\alpha\,S_{{3}}
 \left( k-r \right) -D_{{3}} \left( k-r \right)  \right) N \right)/2}\\
 +& \left( -1 \right) ^{\delta_{{1}}+\delta_{{3}} \left( -k+r-1 \right)/2 }\sum_{\alpha=1}^{\frac{k}{2}} \left( -1 \right) ^{\alpha\,\delta_{{3}}}{q}^{\left(S_{{1}}+D_{{1}}+S_{{3}}
 \left( -k+r+2\,\alpha-1 \right) ^{2}/4+D_{{3}} \left( -k+r+2\,\alpha-1
 \right)/2 \right)/2 }
 \\ \times & \sum _{L=-\infty }^{\infty } \left( -1 \right) ^{ \left(
\delta_{{1}}+\delta_{{2}} +r\delta_{{3}} \right) L}{q}^{\left( krS_{{3}}{L}^{2}+ \left( 2\,
S_{{1}}+D_{{1}}+D_{{2}}+S_{{3}}r \left( -k+r+2\,\alpha-1 \right) +rD_{
{3}} \right) L\right)/2}
\\ \times & \sum _{M=-\infty }^{\infty } \left( -1 \right) ^{ \left( \delta_{{1}}-
\delta_{{2}} \right) M}{q}^{\left(r \left( k-r \right) S_{{3}}{M}^{2}+
 \left( 2\,S_{{1}}+D_{{1}}-D_{{2}} \right) M\right)/2}
\\ \times & \sum _{N=-\infty }^{\infty } \left( -1 \right) ^{ \left( \delta_{{1}}+
\delta_{{2}}- \left( k-r \right) \delta_{{3}} \right) N}{q}^{\left(k \left(
k-r \right) S_{{3}}{N}^{2}+ \left( 2\,S_{{1}}+D_{{1}}+D_{{2}}-S_{{3}}
 \left( k-r \right)  \left( -k+r+2\,\alpha-1 \right) -D_{{3}} \left( k
-r \right)  \right) N\right)/2}
\\+& \left( -1 \right) ^{\delta_{{2}}+\delta_{{3}} \left( k-r+1 \right)/2 }\sum_{\alpha=1}^{\frac{k}{2}} \left( -1 \right) ^{\alpha\,\delta_{{3}}}{q}^{\left(S_{{2}}-D_{{2}}+S_{{3}}
 \left( k-r-2\,\alpha+1 \right) ^{2}/4+D_{{3}} \left( k-r-2\,\alpha+1
 \right)/2 \right)/2 }
\\ \times & \sum _{L=-\infty }^{\infty } \left( -1 \right) ^{ \left(
\delta_{{1}}+\delta_{{2}}+r\delta_{{3}} \right) L}{q}^{\left(krS_{{3}}{L}^{2}+ \left( D_{
{1}}-2\,S_{{2}}+D_{{2}}+S_{{3}}r \left( k-r-2\,\alpha+1 \right) +rD_{{
3}} \right) L\right)/2}
\\ \times &\sum _{M=-\infty }^{\infty } \left( -1 \right) ^{ \left( \delta_{{1}}-
\delta_{{2}} \right) M}{q}^{\left(r \left( k-r \right) S_{{3}}{M}^{2}+
 \left( D_{{1}}+2\,S_{{2}}-D_{{2}} \right) M\right)/2}
\\ \times & \sum _{N=-\infty }^{\infty } \left( -1 \right) ^{ \left( \delta_{{1}}+
\delta_{{2}}- \left( k-r \right) \delta_{{3}} \right) N}{q}^{\left(k \left(
k-r \right) S_{{3}}{N}^{2}+ \left( D_{{1}}-2\,S_{{2}}+D_{{2}}-S_{{3}}
 \left( k-r \right)  \left( k-r-2\,\alpha+1 \right) -D_{{3}} \left( k-
r \right)  \right) N \right)/2}.
\end{align*}
Using  \eqref{2.1} and \eqref{2.1a},  after some simplifications, we obtain
{\small
\begin{align}\label{3.1}
f& \left( \epsilon _{{1}}{q}^{g},\epsilon _{{1}}{q}^{h} \right) f
 \left( \epsilon _{{2}}{q}^{u},\epsilon _{{2}}{q}^{v} \right) f
 \left( \epsilon _{{3}}{q}^{i},\epsilon _{{3}}{q}^{j} \right)
\nonumber \\ =& f_{{\delta_{{1}}+\delta_{{2}}}} \left( {q}^{(r \left( k-r \right)
S_{{3}}+D_{{1}}-D_{{2}})/2},{q}^{(r \left( k-r \right) S_{{
3}}-D_{{1}}+D_{{2}})/2} \right)
 \sum_{\alpha=-{\frac{k-2}{2}}}^{\frac{k}{2}}(-1)^{\alpha \delta_3} q^{\alpha \left({\alpha}S_{{3}}+D_{{3}} \right)/2} \nonumber \\ &\times f_{{\delta_{{1}}+\delta_{{2}}+r\delta_{{3}}}} \left( {q}^{\left(r
 \left( S_{{3}} \left( k+2\, \alpha \right) +D_{{3}} \right) +D_{{1}}
+D_{{2}} \right)/2},{q}^{ \left(r \left( S_{{3}} \left( k-2\, \alpha \right) -D_{
{3}} \right) -D_{{1}}-D_{{2}}\right)/2} \right) \nonumber \\ & \times f_{{\delta_{{1}}+\delta_{{2}}+ \left( k-r \right) \delta_{{3}}}}
 \left( {q}^{\left( \left( k-r \right)  \left( S_{{3}} \left( k-2\,
\alpha \right) -D_{{3}} \right) +D_{{1}}+D_{{2}}\right)/2},{q}^{\left(
 \left( k-r \right)  \left( S_{{3}} \left( k+2\,\alpha \right) +D_{{3}
} \right) -D_{{1}}-D_{{2}}\right)/2} \right) \nonumber\end{align}
 \begin{align}  & + \left( -1 \right) ^{\delta_{{3}} \left( k-r+1 \right)/2 } f_{{\delta_{{1}}+\delta_{{2}}}} \left( {q}^{ \left(2\,r \left( k-r \right) S_
{{3}}+D_{{1}}-D_{{2}}\right)/2},{q}^{\left(-D_{{1}}+D_{{2}}\right)/2} \right)
\nonumber \\ & \times \Bigg\{ \sum_{\alpha=1}^{\frac{k}{2}} (-1)^{\alpha \delta_3+ {\delta_1}} q^{ \left(S_{{1}}+D_{{1}}+\frac{1}{4}\,S_{{3}} \left( -k+r+2\,\alpha-1 \right) ^{2}+\frac{1}{2}
\,D_{{3}} \left( -k+r+2\,\alpha-1 \right)
 \right)/2 }  \nonumber \\ & \times  f_{{\delta_{{1}}+\delta_{{2}}+r\delta_{{3}}}} \left( {q}^{\left(r
 \left( S_{{3}} \left( k+2\,\alpha-1 \right) +D_{{3}} \right) +D_
{{1}}+D_{{2}}\right)/2},{q}^{\left(r \left( S_{{3}} \left( k-2\,\alpha+1
 \right) -D_{{3}} \right) -D_{{1}}-D_{{2}}\right)/2} \right)
 \nonumber \\ & \times  f_{{\delta_{{1}}+\delta_{{2}}+ \left( k-r \right) \delta_{{3}}}}
 \left( {q}^{\left( \left( k-r \right) \left( S_{{3}} \left( 2\,k-2\,\alpha+1 \right) -D_{
{3}} \right)   +D_{{1}}+D_{{2}}\right)/2},{q}^{\left( \left( k-r \right) \left( S_{{3}} \left(2\, \alpha-1 \right) +D_{{3}} \right)  -D_{{1}}-D_{{2}}\right)/2} \right) \nonumber \\ & +
\sum_{\alpha=1}^{\frac{k}{2}} (-1)^{\alpha \delta_3+ {\delta_2}} q^{ \left(S_{{2}}-D_{{2}}+\frac{1}{4}\,S_{{3}} \left( k-r-2\,\alpha+1 \right) ^{2}+\frac{1}{2}
\,D_{{3}} \left( k-r-2\,\alpha+1 \right)
 \right)/2 }
  \nonumber \\ & \times  f_{{\delta_{{1}}+\delta_{{2}}+r\delta_{{3}}}} \left( {q}^{(r
 \left( S_{{3}} \left( k-2\,\alpha+1 \right) +D_{{3}} \right) +D_
{{1}}+D_{{2}})/2},{q}^{(r \left( S_{{3}} \left( k+2\,\alpha-1
 \right) -D_{{3}} \right) -D_{{1}}-D_{{2}})/2} \right)  \nonumber \\ & \times  f_{{\delta_{{1}}+\delta_{{2}}+ \left( k-r \right) \delta_{{3}}}}
 \left( {q}^{( \left( k-r \right)  \left( S_{{3}} \left( 2\,\alpha
-1 \right) -D_{{3}} \right) +D_{{1}}+D_{{2}})/2},{q}^{(
 \left( k-r \right)  \left( S_{{3}} \left( 2\,k-2\,\alpha+1 \right) +D
_{{3}} \right) -D_{{1}}-D_{{2}})/2} \right)
   \Bigg\} .
\end{align}}
\textit{Second Case}. When $\gcd (2k,r)=2 $ and $\gcd(2k,k-r)=1$.\\
In \eqref{3.a1}, we make the change of indices by setting
\begin{equation}\label{3.b2}
  \left( k-r \right) l+ \left( k-r \right) m+2\,n=2k\,L+a , ~~~  l-m=2\,M+b,~~~
  \frac{r}{2}\,l+\frac{r}{2}\,m-n=k\,N+c,
\end{equation}
where $a$  have values selected from the complete set of
residues modulo $2\,k$ and $c$  have values selected from the complete set of
residues modulo $k$. Here, we take $a\in \{0, \pm1,\pm2,\ldots,\pm(k-1),k \}$ and $c\in \{0, \pm1,\pm2,\ldots, \pm\frac{k-1}{2} \}$. The values of $b$ are  selected from the set $\{0,1\}$. Now, solving the simultaneous
system \eqref{3.b2}, we have
\begin{align*}
 l=&L+M+N+{\frac {a+kb+2c}{2k}}, \nonumber \\
 m=&L-M+N+{\frac {a-kb+2c}{2k}},\\
 n=&rL- \left( k-r \right) N+{\frac {ra-2 \left( k-r \right) c}{2k}}.\nonumber
\end{align*}
Since $l,m,n,L,M$ and $N$ are all integers, we see that values of $a$, $b$ and $c$ are associated as follows:
\begin{equation*}
\left\{
  \begin{array}{ll}
   b=0 ~~~\text{and}~~~ c=-\frac{a}{2} , & \hbox{if $a$ is an even;} \\
    b=1~~~ \text{and}~~~ c= \frac{(\text{sign} \, a)\,k-a}{2}, & \hbox{if $a$ is an odd.}
  \end{array}
\right.\end{equation*}
With observing that $ \gcd (k-r,r)=1 $,  $k$ is odd and $r$ is even, thus by arguments similar to those used in the  proof of \textit{First Case}, we obtain
{\small
\begin{align}\label{3.2}
&f \left( \epsilon _{{1}}{q}^{g},\epsilon _{{1}}{q}^{h} \right) f
 \left( \epsilon _{{2}}{q}^{u},\epsilon _{{2}}{q}^{v} \right) f
 \left( \epsilon _{{3}}{q}^{i},\epsilon _{{3}}{q}^{j} \right)\nonumber\\
 = &f_{{\delta_{{1}}+\delta_{{2}}}} \left( {q}^{(r \left( k-r \right)
S_{{3}}+D_{{1}}-D_{{2}})/2},{q}^{(r \left( k-r \right) S_{{
3}}-D_{{1}}+D_{{2}})/2} \right)
 \sum_{\alpha=-{\frac{k-1}{2}}}^{\frac{k-1}{2}}(-1)^{\alpha \delta_3} q^{\alpha \left({\alpha}S_{{3}}+D_{{3}} \right)/2} \nonumber \end{align}
 \begin{align} & \times f_{{\delta_{{1}}+\delta_{{2}}+r\delta_{{3}}}} \left( {q}^{\left(r
 \left( S_{{3}} \left( k+2 \, \alpha \right) +D_{{3}} \right) +D_{{1}}
+D_{{2}} \right)/2},{q}^{ \left(r \left( S_{{3}} \left( k- 2\, \alpha \right) -D_{
{3}} \right) -D_{{1}}-D_{{2}}\right)/2} \right) \nonumber \\ & \times f_{{\delta_{{1}}+\delta_{{2}}+ \left( k-r \right) \delta_{{3}}}}
 \left( {q}^{\left( \left( k-r \right)  \left( S_{{3}} \left( k-2\,
\alpha \right) -D_{{3}} \right) +D_{{1}}+D_{{2}}\right)/2},{q}^{\left(
 \left( k-r \right)  \left( S_{{3}} \left( k+2\,\alpha \right) +D_{{3}
} \right) -D_{{1}}-D_{{2}}\right)/2} \right) \nonumber \\ & + \left( -1 \right) ^{\delta_{{3}} \left( k-r+1 \right)/2 } f_{{\delta_{{1}}+\delta_{{2}}}} \left( {q}^{ \left(2\,r \left( k-r \right) S_
{{3}}+D_{{1}}-D_{{2}}\right)/2},{q}^{\left(-D_{{1}}+D_{{2}}\right)/2} \right)
\nonumber \\ & \times \Bigg\{ \sum_{\alpha=1}^{\frac{k+1}{2}} (-1)^{\alpha \delta_3+ {\delta_1}} q^{ \left(S_{{1}}+D_{{1}}+\frac{1}{4}\,S_{{3}} \left( -k+r+2\,\alpha-1 \right) ^{2}+\frac{1}{2}
\,D_{{3}} \left( -k+r+2\,\alpha-1 \right)
 \right)/2 }  \nonumber \\ & \times  f_{{\delta_{{1}}+\delta_{{2}}+r\delta_{{3}}}} \left( {q}^{\left(r
 \left( S_{{3}} \left( k+2\,\alpha-1 \right) +D_{{3}} \right) +D_
{{1}}+D_{{2}}\right)/2},{q}^{\left(r \left( S_{{3}} \left( k-2\,\alpha+1
 \right) -D_{{3}} \right) -D_{{1}}-D_{{2}}\right)/2} \right)
 \nonumber \\ & \times  f_{{\delta_{{1}}+\delta_{{2}}+ \left( k-r \right) \delta_{{3}}}}
 \left( {q}^{\left( \left( k-r \right) \left( S_{{3}} \left( 2\,k-2\,\alpha+1 \right) -D_{
{3}} \right)   +D_{{1}}+D_{{2}}\right)/2},{q}^{\left( \left( k-r \right) \left( S_{{3}} \left(2\, \alpha-1 \right) +D_{{3}} \right)   -D_{{1}}-D_{{2}}\right)/2} \right) \nonumber \\ & +
\sum_{\alpha=1}^{\frac{k-1}{2}} (-1)^{\alpha \delta_3+ {\delta_2}} q^{ \left(S_{{2}}-D_{{2}}+\frac{1}{4}\,S_{{3}} \left( k-r-2\,\alpha+1 \right) ^{2}+\frac{1}{2}
\,D_{{3}} \left( k-r-2\,\alpha+1 \right)
 \right)/2 }
  \nonumber \\ & \times  f_{{\delta_{{1}}+\delta_{{2}}+r\delta_{{3}}}} \left( {q}^{(r
 \left( S_{{3}} \left( k-2\,\alpha+1 \right) +D_{{3}} \right) +D_
{{1}}+D_{{2}})/2},{q}^{(r \left( S_{{3}} \left( k+2\,\alpha-1
 \right) -D_{{3}} \right) -D_{{1}}-D_{{2}})/2} \right)  \nonumber \\ & \times  f_{{\delta_{{1}}+\delta_{{2}}+ \left( k-r \right) \delta_{{3}}}}
 \left( {q}^{( \left( k-r \right)  \left( S_{{3}} \left( 2\,\alpha
-1 \right) -D_{{3}} \right) +D_{{1}}+D_{{2}})/2},{q}^{(
 \left( k-r \right)  \left( S_{{3}} \left( 2\,k-2\,\alpha+1 \right) +D
_{{3}} \right) -D_{{1}}-D_{{2}})/2} \right)
   \Bigg\} .
\end{align}}
Since $k-r$ is odd in both cases, so $\delta_1+\delta_2+ (k-r) \delta_3 \equiv \delta_1+\delta_2+  \delta_3 \pmod{2}$. Thus
combining identity \eqref{3.1} with \eqref{3.2}  together, we deduce \eqref{3.3}. This completes the proof of the theorem.
\end{proof}
\begin{corollary}\label{cor1}
Let $k$ and $r$ be positive integers with $k>r$, $\gcd (2k,r)=1 \, \text{or} \, 2 $ and $\gcd(2k,k-r)=1$, we have
\begin{align*}
  8 \psi^{2} & \left( q^{r(k-r)} \right)\psi(q^2)=\varphi \left( q^{r(k-r)} \right) \sum_{\alpha={\left[\frac{2 - k}{2}\right]}}^{\left[\frac{k}{2}\right]} q^{\alpha (\alpha+1)} f \left(  q^{r(2\,k-r+ 2\alpha + 1)} , q^{r(r- 2\alpha - 1)}   \right)\\
& \times f \left(  q^{(k-r)(k+r- 2\,\alpha - 1)} , q^{(k-r)(k-r+ 2\alpha + 1)}   \right)
+ 2 \psi \left( q^{2r(k-r)} \right)\\& \Bigg\{  \sum_{\alpha=1}^{\left[\frac{k+1}{2}\right]} q^{r(k-r)+\frac{1}{4}\left( -k+r+2\alpha-1 \right) ^{2}+\frac{1}{2}
\left( -k+r+2\alpha-1 \right)}
 f \left(  q^{r(2\,k-r+ 2\alpha )} , q^{r(r- 2\alpha )}   \right)\\ & \times
f \left(  q^{(k-r)(2k+r- 2\alpha )} , q^{(k-r)(-r+ 2\alpha )}   \right)+ \sum_{\alpha=1}^{\left[\frac{k}{2}\right]} q^{\frac{1}{4} \left( k-r-2\alpha+1 \right) ^{2}+\frac{1}{2}
\left( k-r-2\alpha+1 \right)}
\\ & \times f \left(  q^{r(2\,k - r- 2\alpha + 2 )} , q^{r(r+ 2\alpha -2 )}   \right)
f \left(  q^{(k-r)(r+ 2\alpha -2 )} , q^{(k-r)(2\,k-r- 2\alpha+2 )}   \right)\Bigg\}.
\end{align*}
\end{corollary}
\begin{proof}
Setting $g=u=r(k-r)$, $h=v=0$, $i=2$, $j=0$ and $\epsilon_1=\epsilon_2=\epsilon_3=1$ in Theorem \ref{MainThm3}, after some simplifications, we obtain the desired result.
\end{proof}

In a similar way, we obtain the following corollaries:

\begin{corollary}\label{cor2}
Under the same conditions as Corollary $\ref{cor1}$, we have
\begin{align*}
  4 \, \psi^{2} & \left( q^{r(k-r)} \right) \varphi(q) =\varphi \left( q^{r(k-r)} \right) \sum_{\alpha={\left[\frac{2 - k}{2}\right]}}^{\left[\frac{k}{2}\right]}  q^{\alpha^2} f\left(  q^{r(2\,k-r+2\,\alpha)},q^{r(r-2\,\alpha)}  \right)\\
&\times f \left(  q^{(k-r)(k+r- 2\,\alpha )} , q^{(k-r)(k-r+ 2\,\alpha )}   \right)+ 2\,  \psi \left( q^{2\,r(k-r)} \right) \Bigg\{  \sum_{\alpha=1}^{\left[\frac{k+1}{2}\right]}  q^{r(k-r)+\frac{1}{4}\, \left( -k+r+2\,\alpha-1 \right) ^{2}}\\
 & \times f \left(  q^{r(2\,k-r+ 2\,\alpha -1)} , q^{r(r- 2\,\alpha +1 )}   \right)
f \left(  q^{(k-r)(2\,k+r- 2\,\alpha +1)} , q^{(k-r)(-r+ 2\,\alpha -1)}   \right)\\
&+ \sum_{\alpha=1}^{\left[\frac{k}{2}\right]}  q^{\frac{1}{4}\, \left( k-r-2\,\alpha+1 \right) ^{2}}
f \left(  q^{r(2\,k - r- 2\,\alpha + 1 )} ,  q^{r(r+ 2\,\alpha -1 )}   \right) \times f \left(  q^{(k-r)(r+ 2\,\alpha -1 )} , q^{(k-r)(2\,k-r- 2\,\alpha+1 )}   \right)\Bigg\}.
\end{align*}
\end{corollary}

\begin{corollary}\label{cor3}
Under the same conditions as Corollary $\ref{cor1}$, we have
\begin{align*}
\varphi^2 & \left( q^{r(k-r)} \right)\varphi(q^2)=\varphi \left( q^{2r(k-r)} \right) \sum_{\alpha={\left[\frac{2 - k}{2}\right]}}^{\left[\frac{k}{2}\right]} q^{2 \, \alpha^2} f\left(  q^{2\,r(k+2\,\alpha)},q^{2\, r(k-2\,\alpha)}  \right) f \left( q^{2\,(k-r)(k- 2\,\alpha )} , q^{2\,(k-r)(k+ 2\,\alpha )}   \right)\\& + 2\, q^{r(k-r)} \psi \left( q^{4\,r(k-r)} \right) \Bigg\{  \sum_{\alpha=1}^{\left[\frac{k+1}{2}\right]} q^{\frac{1}{2}\, \left( -k+r+2\,\alpha-1 \right) ^{2}}   f \left(  q^{2\,r(k+ 2\,\alpha -1)} , q^{2\,r(k - 2\,\alpha +1 )}   \right)\\
& \times
f \left(  q^{2\,(k-r)(2\,k- 2\,\alpha +1)} , q^{2\,(k-r)( 2\,\alpha -1)}   \right)+ \sum_{\alpha=1}^{\left[\frac{k}{2}\right]}q^{\frac{1}{2}\, \left( k-r-2\,\alpha+1 \right) ^{2}}
f \left(  q^{2\,r(k - 2\,\alpha + 1 )} ,  q^{2\,r(k+ 2\,\alpha -1 )}   \right)\\
& \times f \left(  q^{2\,(k-r)( 2\,\alpha -1 )} , q^{2\,(k-r)(2\,k- 2\,\alpha+1 )}   \right)\Bigg\}.
\end{align*}
\end{corollary}

\begin{corollary}\label{cor4} Under the same conditions as Corollary $\ref{cor1}$, we have

  \begin{align*}
\varphi^2 & \left( q^{r(k-r)} \right)\psi(q)=\varphi \left( q^{2r(k-r)} \right) \sum_{\alpha={\left[\frac{2 - k}{2}\right]}}^{\left[\frac{k}{2}\right]} q^{\alpha(2\, \alpha + 1)} f\left(  q^{r(2\,k+4\,\alpha+1)},q^{r(2\,k-4\,\alpha-1)}  \right) \\& \times  f \left( q^{(k-r)(2\,k- 4\,\alpha-1 )} , q^{(k-r)(2\,k+ 4\,\alpha+1 )}   \right) + 2\,q^{r(k-r)}  \psi \left( q^{4\,r(k-r)} \right)\\&\times 
 \Bigg\{  \sum_{\alpha=1}^{\left[\frac{k+1}{2}\right]} q^{\frac{1}{2}\, \left( -k+r+2\,\alpha-1 \right) ^{2}+\frac{1}{2}\, \left( -k+r+2\,\alpha-1 \right)} 
  f \left(  q^{r(2\,k+ 4\,\alpha -1)} , q^{r(2\,k - 4\,\alpha +1 )}   \right)\\&\times
 f \left(  q^{(k-r)(4\,k- 4\,\alpha +1)} , q^{(k-r)( 4\,\alpha -1)}   \right)
 + \sum_{\alpha=1}^{\left[\frac{k}{2}\right]}q^{\frac{1}{2}\, \left( k-r-2\,\alpha+1 \right) ^{2}+\frac{1}{2}\, \left( k-r-2\,\alpha+1 \right)}\\
 & \times
 f \left(  q^{r(2\,k - 4\,\alpha + 3 )} ,  q^{r(2k+ 4\,\alpha -3 )}   \right) f \left(  q^{(k-r)( 4\,\alpha -3 )} , q^{(k-r)(4\,k- 4\,\alpha+3 )}   \right)\Bigg\}.
\end{align*}

\end{corollary}


\section{Product of two  theta functions} \label{pro2thetaF}
Our main theorem can be considered as a generalisation for the product of two theta functions and one can obtain several identities, for instance
\begin{corollary} \label{clp2}  For any positive integer $m$, we have
\begin{align*}
  \varphi  \left( -{q}^{m} \right) \varphi  \left( q \right) = & \sum _{
\alpha=\left[\frac{1-m}{2} \right]}^{\left[\frac{1+m}{2} \right]}{q}^{{\alpha}^{2}}f \left( -{q}^{m
 \left( m+1+2\,\alpha \right) },-{q}^{m \left( m+1-2\,\alpha \right) }
 \right)f \left( -{q}^{m+1-2\,\alpha},-{q}^{m+1+2\,\alpha} \right),\\
\psi \left( -{q}^{m} \right) \psi \left( q \right) =& \sum _{
\alpha=\left[\frac{1-m}{2} \right]}^{\left[\frac{1+m}{2} \right]}{q}^{2\,{\alpha}^{2}+\alpha}f \left( -{q}^{2\,m
 \left( m+2\,\alpha+2 \right) },-{q}^{2\,m \left( m-2\,\alpha \right)
} \right) f \left( -{q}^{3\,m+1-4\,\alpha},-{q}^{m+3+4\,\alpha}
 \right),\\
f \left( -{q}^{m} \right) f \left( -q \right) =& \sum _{
\alpha=\left[\frac{1-m}{2} \right]}^{\left[\frac{1+m}{2} \right]} \left( -1 \right) ^{\alpha}{q}^{\alpha\, \left(
3\,\alpha+1 \right)/2 }f_{{m+1}} \left( {q}^{m \left( 3\,m+6\,
\alpha+5 \right)/2 },{q}^{m \left( 3\,m-6\,\alpha+1 \right)/2 }
 \right)\\&\times f \left( {q}^{2\,m+1-3\,\alpha},{q}^{2+m+3\,\alpha} \right),\\
2\,\varphi  \left( -{q}^{m} \right) \psi \left( {q}^{2} \right) = & \sum _{
\alpha=\left[\frac{1-m}{2} \right]}^{\left[\frac{1+m}{2} \right]}{q}^{{\alpha}^{2}+\alpha}f \left( -{
q}^{m \left( m+2\,\alpha+2 \right) },-{q}^{m \left( m-2\,\alpha
 \right) } \right) f \left( -{q}^{m-2\,\alpha},-{q}^{m+2\,\alpha+2}
 \right),\\
\varphi  \left( -{q}^{2\,m} \right) \psi \left( q \right) =& \sum _{
\alpha=\left[\frac{1-m}{2} \right]}^{\left[\frac{1+m}{2} \right]}{q}^{2\,{\alpha}^{2}+\alpha}f \left( -
{q}^{m \left( 2\,m+3+4\,\alpha \right) },-{q}^{m \left( 2\,m+1-4\,
\alpha \right) } \right) f \left( -{q}^{2\,m+1-4\,\alpha},-{q}^{2\,m+3
+4\,\alpha} \right),\\
f \left( -{q}^{2\,m} \right) \varphi  \left( {q}^{3} \right) =& \sum _{
\alpha=\left[\frac{1-m}{2} \right]}^{\left[\frac{1+m}{2} \right]}{q}^{3\,{\alpha}^{2}}f \left( -{q}^{m
 \left( 3\,m+6\,\alpha+4 \right) },-{q}^{m \left( 3\,m-6\,\alpha+2
 \right) } \right) f \left( -{q}^{4\,m+3-6\,\alpha},-{q}^{2\,m+3+6\,
\alpha} \right),\\
\psi \left( -{q}^{m} \right) \varphi  \left( {q}^{2} \right) =& \sum _{
\alpha=\left[\frac{1-m}{2} \right]}^{\left[\frac{1+m}{2} \right]}{q}^{2\,{\alpha}^{2}}f \left( -{q}^{m
 \left( 2\,m+3+4\,\alpha \right) },-{q}^{m \left( 2\,m+1-4\,\alpha
 \right) } \right)  f \left( -{q}^{3\,m+2-4\,\alpha},-{q}^{m+2+4\,
\alpha} \right),\\
2\,\psi \left( -{q}^{m} \right) \psi \left( {q}^{4} \right) =& \sum _{
\alpha=\left[\frac{1-m}{2} \right]}^{\left[\frac{1+m}{2} \right]}{q}^{2\,{\alpha}^{2}+2\,\alpha}f
 \left(- {q}^{m \left( 2\,m+4\,\alpha+5 \right) },-{q}^{m \left( 2\,m-4
\,\alpha-1 \right) } \right) f \left(- {q}^{3\,m-4\,\alpha},-{q}^{
m+4+4\,\alpha} \right).
\end{align*}
\end{corollary}
\begin{proof}
Sitting $k=m+1$ and $r=m$ in Theorem \ref{MainThm3}, we observe that
\begin{equation*}
  \begin{cases}
\gcd (2\,m+2,m)=1~~\text{and} ~~\gcd (2\,m+2,1)=1 & \text{if } m \text{ is odd} ,\\
\text{or}
\\
\gcd (2\,m+2,m)=2~~\text{and} ~~\gcd (2\,m+2,1)=1 & \text{if } m \text{ is even}.
\end{cases}
\end{equation*}
We set the values of $g$, $h$, $u$, $v$, $i$, $j$, $\epsilon_1$, $\epsilon_2$ and $\epsilon_3$ as in the following table in the same order as in the corollary:\\
\begin{center}
\begin{tabular}{|c|c|c|c|c|c|c|c|c|}
  \hline
  $g$ & $h$ & $u$ & $v$ & $i$ & $j$ & $\epsilon_1$ & $\epsilon_2$ & $\epsilon_3$ \\\hline \hline
  $m$ & $m$ & $m$ & $m$ & $2$ & $2$ & $-1$ & $1$ & $1$ \\\hline
  $3\,m$ & $m$ & $3\,m$ & $m$ & $6$ & $2$ & $-1$ & $1$ & $1$ \\\hline
  $2\,m$ & $m$ & $2\,m$ & $m$ & $4$ & $2$ & $-1$ & $1$ & $-1$ \\\hline
  $m$ & $m$ & $m$ & $m$ & $4$ & $0$ & $-1$ & $1$ & $1$ \\\hline
  $m$ & $m$ & $m$ & $m$ & $3$ & $1$ & $-1$ & $1$ & $1$ \\\hline
  $2\,m$ & $m$ & $2\,m$ & $m$ & $3$ & $3$ & $-1$ & $1$ & $1$ \\ \hline
  $3\,m$ & $m$ & $3\,m$ & $m$ & $4$ & $4$ & $-1$ & $1$ & $1$ \\\hline
  $3\,m$ & $m$ & $3\,m$ & $m$ & $8$ & $0$ & $-1$ & $1$ & $1$ \\
  \hline
\end{tabular}
\end{center}
Then, we apply  \eqref{2.5} and the identity \cite[Entry 30(iv)]{Adiga3}
\begin{equation} f(a,b)f(-a,-b)=f(-a^2,-b^2)\varphi(-ab).\nonumber \end{equation}
Finally,  replacing $q^2$ by $q$, after some simplifications, we deduce the desired results.
\end{proof}

The identities appearing in Corollary \ref{clp2} are analogous  to some identities found in the literature. For example, see  \cite{SCHEN}, \cite{Yan}, \cite{Cao} and  \cite{Adiga2}. Using this corollary, one can obtain several theta function identities found in Ramanujan's Notebook \cite{Ramanujan1} and  many modular relations for the Rogers-Ramanujan type functions.


Using our main theorem and Lemma \ref{lemma1}, we prove the following theorem which contains a general  formula  for the product of two theta functions:

\begin{theorem}\label{thm2} Let $k$ and $r$ be positive integers with $k>r$, $\gcd (2k,r)=1 \, \text{or} \, 2 $ and $\gcd(2k,k-r)=1$. For any integers $s$, $t$, $i$ and $j$ with
$s+t=S>0$, $s-t=D $,  $i+j=S_3>0$, $i-j=D_3$,  and $S=r \left( k-r \right) S_{{3}}$, we have
\begin{align}\label{thm2eq1}
  f \left( {q}^{s},{q}^{t} \right)  f
 \left( \epsilon {q}^{i},\epsilon{q}^{j} \right)
 \nonumber
=& \sum_{\alpha={\left[\frac{2 - k}{2}\right]}}^{\left[\frac{k}{2}\right]}(-1)^{\alpha \delta} q^{\alpha \left({\alpha}S_{{3}}+D_{{3}} \right)/2}
 \nonumber \\ & \times
  f_{r\delta} \left( {q}^{\left(r
 \left( S_{{3}} \left( k+2 \, \alpha \right) +D_{{3}} \right) +D \right)/2},{q}^{ \left(r \left( S_{{3}} \left( k- 2\, \alpha \right) -D_{
{3}} \right) -D\right)/2} \right) \nonumber \\ & \times f_{{ \delta}}
 \left( {q}^{\left( \left( k-r \right)  \left( S_{{3}} \left( k-2\,
\alpha \right) -D_{{3}} \right) +D\right)/2},{q}^{\left(
 \left( k-r \right)  \left( S_{{3}} \left( k+2\,\alpha \right) +D_{{3}
} \right) -D\right)/2} \right).
\end{align}
Here $\epsilon = \pm1$,   $\delta=
\ds\frac{1-\epsilon}{2}$  and $[x]$ denote the
greatest integer less than or equal to $x$.
\end{theorem}
\begin{proof} Setting $\epsilon_{1}=1$ and $\epsilon_{2}=1$ in Theorem \ref{MainThm3},
 again setting $\epsilon_{1}=-1$ and $\epsilon_{2}=-1$ in Theorem \ref{MainThm3},
 adding the resulting identities and then applying Lemma \ref{lemma1} with
  $a=q^g$, $b=q^h$, $c=q^u$ and $d=q^v$, after some simplifications, we obtain
\begin{align*}
  f &\left( {q}^{g+u},{q}^{h+v} \right)   f
 \left( \epsilon _{{3}}{q}^{i},\epsilon _{{3}}{q}^{j} \right)=
 \sum_{\alpha={\left[\frac{2 - k}{2}\right]}}^{\left[\frac{k}{2}\right]}(-1)^{\alpha \delta_3} q^{\alpha \left({\alpha}S_{{3}}+D_{{3}} \right)/2} \nonumber \\ & \times f_{{r\delta_{{3}}}} \left( {q}^{\left(r
 \left( S_{{3}} \left( k+2 \, \alpha \right) +D_{{3}} \right) +D_{{1}}
+D_{{2}} \right)/2},{q}^{ \left(r \left( S_{{3}} \left( k- 2\, \alpha \right) -D_{
{3}} \right) -D_{{1}}-D_{{2}}\right)/2} \right) \nonumber \\ & \times f_{{ \delta_{{3}}}}
 \left( {q}^{\left( \left( k-r \right)  \left( S_{{3}} \left( k-2\,
\alpha \right) -D_{{3}} \right) +D_{{1}}+D_{{2}}\right)/2},{q}^{\left(
 \left( k-r \right)  \left( S_{{3}} \left( k+2\,\alpha \right) +D_{{3}
} \right) -D_{{1}}-D_{{2}}\right)/2} \right).
\end{align*}
Putting $s=g+u$, $t=h+v$, $S=s+t$ and $D=s-t$, so that $S=S_1+S_2=2\,S_1$, $D=D_1+D_2$ and $S=r(k-r)S_3$, using these in the above identity, with replacing $\epsilon_3$ and $\delta_3$ by $\epsilon$ and $\delta$, respectively, we deduce \eqref{thm2eq1}.
\end{proof}
\section{Applications to the theory of representation of integers}\label{Applications}
In this section, we use Theorem \ref{MainThm3} to establish many relations between the  number of representations of $N \in \mathbb{N}_0$ as  ternary sums involving squares, triangular numbers, generalized  pentagonal
numbers and generalized octagonal numbers. We define for  $N \in \mathbb{N}_0$, the  following:
\begin{align*}
  r(a_1,a_2,a_3;N):=& \left|  \left\{ (l,m,n) \in \mathbb{Z}^3 | N= a_1\,l^2 +a_2 \,m^2+a_3n^2  \right\}  \right|,\\
   T(a_1,a_2,a_3;N):=& \left|  \left\{ (l,m,n) \in \mathbb{N}_0^3| N= a_1\,t_{l} +a_2 \,t_{m}+a_3t_{n}  \right\}  \right|,\\
     P(a_1,a_2,a_3;N):=& \left|  \left\{ (l,m,n) \in \mathbb{Z}^3 | N= a_1\,p_{l} +a_2 \,p_{m}+a_3p_{n}  \right\}  \right|,\\
     G(a_1,a_2,a_3;N):=& \left|  \left\{ (l,m,n) \in \mathbb{Z}^3 | N= a_1\,g_{l} +a_2 \,g_{m}+a_3g_{n}  \right\}  \right|,\\
     Rt(a_1,a_2,a_3;N):=& \left|  \left\{ (l,m,n) \in \mathbb{Z}^2\times \mathbb{N}_0 | N= a_1\,l^2 +a_2 \,m^2+a_3\,t_n  \right\}  \right|,\\
Rp(a_1,a_2,a_3;N):=& \left|  \left\{ (l,m,n) \in \mathbb{Z}^3 | N= a_1\,l^2 +a_2 \,m^2+a_3\,p_n  \right\}  \right|,\\
     Rg(a_1,a_2,a_3;N):=& \left|  \left\{ (l,m,n) \in \mathbb{Z}^3 | N= a_1\,l^2 +a_2 \,m^2+a_3\,g_n  \right\}  \right|,\\
     Tp(a_1,a_2,a_3;N):=& \left|  \left\{ (l,m,n) \in \mathbb{N}^2_0 \times \mathbb{Z} | N= a_1\,t_l +a_2 \,t_m+a_3\,p_n  \right\}  \right|,\\
     Tg(a_1,a_2,a_3;N):=& \left|  \left\{ (l,m,n) \in \mathbb{N}^2_0 \times \mathbb{Z} | N= a_1\,t_l +a_2 \,t_m+a_3\,g_n  \right\}  \right|,\\
     rT(a_1,a_2,a_3;N):=& \left|  \left\{ (l,m,n) \in \mathbb{Z}\times \mathbb{N}^2_0 | N= a_1\,l^2 +a_2 \,t_m+a_3\,t_n  \right\}  \right|,\\
     rP(a_1,a_2,a_3;N):=& \left|  \left\{ (l,m,n) \in \mathbb{Z}^3 | N= a_1\,l^2 +a_2 \,p_m+a_3\,p_n  \right\}  \right|,\\
     rG(a_1,a_2,a_3;N):=& \left|  \left\{ (l,m,n) \in \mathbb{Z}^3 | N= a_1\,l^2 +a_2 \,g_m+a_3\,g_n  \right\}  \right|,\\
     pG(a_1,a_2,a_3;N):=& \left|  \left\{ (l,m,n) \in \mathbb{Z}^3 | N= a_1\,p_l +a_2 \,g_m+a_3\,g_n  \right\}  \right|,\\
     tP(a_1,a_2,a_3;N):=& \left|  \left\{ (l,m,n) \in \mathbb{N}_0 \times \mathbb{Z}^2 | N= a_1\,t_l +a_2 \,p_m+a_3\,p_n  \right\}  \right|,\\
     tG(a_1,a_2,a_3;N):=& \left|  \left\{ (l,m,n) \in \mathbb{N}_0 \times \mathbb{Z}^2 | N= a_1\,t_l +a_2 \,g_m+a_3\,g_n  \right\}  \right|,\\
     Pg(a_1,a_2,a_3;N):=& \left|  \left\{ (l,m,n) \in \mathbb{Z}^3 | N= a_1\,p_l +a_2 \,p_m+a_3\,g_n  \right\}  \right|,\\
     rtp(a_1,a_2,a_3;N):=& \left|  \left\{ (l,m,n) \in \mathbb{Z}\times \mathbb{N}_0\times\mathbb{Z} | N= a_1\,l^2 +a_2 \,t_m+a_3\,p_n  \right\}  \right|,\\
     rtg(a_1,a_2,a_3;N):=& \left|  \left\{ (l,m,n) \in \mathbb{Z}\times \mathbb{N}_0\times\mathbb{Z} | N= a_1\,l^2 +a_2 \,t_m+a_3\,g_n  \right\}  \right|,
     \end{align*}
     \begin{align*}
     rpg(a_1,a_2,a_3;N):=& \left|  \left\{ (l,m,n) \in \mathbb{Z}^3\mathbb{Z} | N= a_1\,l^2 +a_2 \,p_m+a_3\,g_n  \right\}  \right|,\\
     tpg(a_1,a_2,a_3;N):=& \left|  \left\{ (l,m,n) \in \mathbb{N}_0 \times \mathbb{Z}^2 | N= a_1\,t_l +a_2 \,p_m+a_3\,g_n  \right\}  \right|.
\end{align*}
where $a_1, a_2,  a_3 \in \mathbb{N}$.
\begin{theorem}\label{Athm1}
For each non-negative integer $N$, we have
\begin{equation*}
rT(1,1,1;N) =
\begin{cases}
Rt(2,2,2;N) & \text{if } N \equiv 0\pmod{2},
\\
4 \, T(2,4,4;N-1) & \text{if } N \equiv 1\pmod{2}.
\end{cases}
\end{equation*}
\end{theorem}

\begin{proof}Putting $k=2$, $r=1$, $\epsilon_1=\epsilon_2=\epsilon_3=1$, $g=u=1$, $h=v=0$ and $i=j=1$ in Theorem \ref{MainThm3}, we deduce
\begin{equation}
  \label{cor51}   \psi^{2} \left( q \right)   \varphi  \left( q \right) =
\psi \left( {q}^{2} \right)   \varphi^{2}  \left( {q}^
{2} \right) +4\,q\psi \left( {q}^{2} \right)   \psi^{2} \left( {q}^{4} \right).
\end{equation}
 By the definitions of $\varphi(q)$ and $\psi(q)$, we can rewrite \eqref{cor51} in the form
\begin{align*}
  \left(\sum_{l=-\infty}^{\infty}q^{l^2}  \right)\left(\sum_{m=0}^{\infty} q^{t_m} \right ) \left(\sum_{n=0}^{\infty} q^{t_n} \right) =&\left(\sum_{l=-\infty}^{\infty} q^{2\,l^2} \right) \left(\sum_{m=-\infty}^{\infty}q^{2\, m^2}  \right)\left(\sum_{n=0}^{\infty} q^{2\,t_n} \right)\\
&+4\,q \left(\sum_{l=0}^{\infty} q^{2\,t_l} \right) \left(\sum_{m=0}^{\infty}q^{4\,t_m}  \right)\left(\sum_{n=0}^{\infty} q^{4\,t_n} \right).
\end{align*}
Hence
\begin{equation}\label{Ath11}
  \left(\sum_{N=l^2+t_m+t_n}1  \right)=\left( \sum_{N=2(l^2+m^2+t_n)}1 \right)+4\left( \sum_{N-1=2(t_l+2(t_m+t_n))}1 \right),
\end{equation}
which can be written in the form
\begin{equation} \label{Ath12} rT(1,1,1;N) =Rt(2,2,2;N) + 4 \, T(2,4,4;N-1).\end{equation}
From \eqref{Ath11}, we  observe that if $N$ is even then $N-1$ is odd and  cannot be represented in the form $2(t_l+2(t_m+t_n))$, so that  $T(2,4,4;N-1)=0$. Similarly, if $N$ is odd then $N$ cannot be  represented in the form  $2(l^2+m^2+t_n)$, so that  $Rt(2,2,2;N)=0$. Thus, using \eqref{Ath12},  we deduce the desired result.
\end{proof}
\begin{exm}
We have $rT(1,1,1;5)=8$, because there are $8$ ways to represent $5$ in the form $l^2+t_m+t_n$, namely
 $(2)^2+0+1$, $(-2)^2+0+1$, $(2)^2+1+0$, $(-2)^2+1+0$, $(1)^2+1+3$, $(-1)^2+1+3$, $(1)^2+3+1$, $(-1)^2+3+1$,   $Rt(2,2,2;5)=0$ because $5$ cannot be represented in the form $2(l^2+m^2+t_n)$, and $T(2,4,4;4)=2$ because there are $2$ ways to represent $4$ in the form $2(t_l+2(t_m+t_n))$, namely, $2(0+2(1)+2(0))$, $2(0+2(0)+2(1))$ so that $rT(1,1,1;5) =4\, T(2,4,4;4) $. One can check that $rT(1,1,1;10) =Rt(2,2,2;10)=16$.
 \end{exm}
\begin{remark}
Theorem \ref{Athm1} can be written in the following form: For each non-negative integer $N$, we have
\begin{align*}
 rT(1,1,1;2\,N)& =Rt(1,1,1;N)\\
  rT(1,1,1;2\,N)& = 4 \, T(1,2,2;N)
\end{align*}
\end{remark}
\begin{corollary}
Every non-negative integer  can be written in the form
$l^2+m^2+t_n ~~ \text{or}~~ t_l+2\,t_m+2\,t_n $.
\end{corollary}
\begin{proof}
Sun \cite{Sun} proved that, every non-negative integer $N$ can be written in the form $l^2+t_m+t_n$, so that $rT(1,1,1;N) \neq 0$. Thus
 \begin{equation*}
0\neq
\begin{cases}
Rt(2,2,2;N) & \text{if } N \equiv 0\pmod{2},
\\
4 \, T(2,4,4;N-1) & \text{if } N \equiv 1\pmod{2}.
\end{cases}
\end{equation*}
Then, for any non-negative integer $k$, we have $Rt(2,2,2;2\,k)\neq 0$, so that there exist $l,m \in \mathbb{Z}$ and $n \in \mathbb{N}_0$ such that $2\,k=2(l^2+m^2+t_n)$ hence  $k=l^2+m^2+t_n$, since $k\in \mathbb{N}_0$ is arbitrary, this completes the first part of the corollary. In a similar way, we deduce the second part.
\end{proof}

First part of the above corollary is established by Sun \cite{Sun}. Second part is due to  Liouville.



\begin{theorem}\label{Athm2}
For each non-negative integer $N$, we have
\begin{equation*}
Rt(1,1,4;N) =
\begin{cases}
rT(2,2,2;N) & \text{if } N \equiv 0\pmod{2},
\\
4 \, rT(4,4,8;N-1) & \text{if } N \equiv 1\pmod{4},
\\ 0    &  \text{if } N \equiv 3\pmod{4}.
\end{cases}
\end{equation*}
\end{theorem}
\begin{proof}
Setting $k=2$, $r=1$, $\epsilon_1=\epsilon_2=\epsilon_3=1$, $g=h=u=v=1$, $i=4$ and $j=0$ in Theorem \ref{MainThm3}, we obtain
\begin{equation}
  \label{cor52} \varphi^{2}  \left( q \right)  \psi \left( {q}^{4}
 \right) =\varphi  \left( {q}^{2} \right)   \psi^{2} \left( {q}^{2}
 \right)  +4\,q\psi \left( {q}^{4} \right) \psi \left( {q}
^{8} \right) \varphi  \left( {q}^{4} \right).
\end{equation}
Using \eqref{cor52}, we deduce
\begin{equation}\label{Ath21}
\left(  \sum_{N=l^2+m^2+4\,t_n}1 \right)=\left(  \sum_{N=2(l^2+t_m+t_n)}1 \right)+4\left(  \sum_{N-1=4(l^2+2\,t_m+t_n)}1 \right),
\end{equation}
which can be written as
\begin{equation} \label{Ath22} Rt(1,1,4;N) =  rT(2,2,2;N) + 4 \, rT(4,4,8;N-1).\end{equation}
From \eqref{Ath21}, we observe that, if $N$ is even then  $N-1$ is odd and cannot be represented in the form  $4(l^2+2\,t_m+t_n)$, so that  $rT(4,4,8;N-1)=0$. If $N$ is odd then,  it cannot be written in the form $2(l^2+t_m+t_n)$, so that  $rT(2,2,2;N)=0$. But in case $N$ is odd it can be written in either $4\,k+1$ or $4\,k+3$, if $N=4\,k+3$, then $N-1$ cannot be written in the form $4(l^2+2\,t_m+t_n)$, and so $rT(4,4,8;N-1)=0$. Thus, using \eqref{Ath22},  we obtain the desired result.
\end{proof}
\begin{remark}
We can rewrite Theorem \ref{Athm2} in the following form:
For each non-negative integer $N$, we have
\begin{align*}
 Rt(1,1,4;2\,N) &=  rT(1,1,1;N),\\
 Rt(1,1,4;4\,N+1) &=  4 \, rT(1,1,2;N),\\
  Rt(1,1,4;4\,N+3) &=0.
\end{align*}
\end{remark}

By the fact that,  every non-negative integer $N$ can be written in the form $l^2+t_m+t_n$ or in the form $l^2+2\,t_m+t_n$, which was proved by Sun \cite{Sun}, and using Theorem \ref{Athm2}, we deduce the following result:
\begin{corollary} Every non-negative integer $N$ can be represented in the form $l^2+m^2+4\,t_n$ ($l,m,n \in \mathbb{Z})$),
if and only if it is not of the form $4\,k+3$ $(k \in \mathbb{N}_0)$.
\end{corollary}


\begin{theorem}
\label{Athm3}For each non-negative integer $N$, we have
\begin{equation*}
r(1,1,2;N) =
\begin{cases}
r(2,4,4;N)+4\,rT(2,8,8;N-2)  & \text{if } N \equiv 0\pmod{2},
\\
4 \, T(2,2,4;N-1) & \text{if } N \equiv 1\pmod{2}.
\end{cases}
\end{equation*}
\end{theorem}
\begin{proof}
 Setting $k=2$, $r=1$, $\epsilon_1=\epsilon_2=\epsilon_3=1$, $g=h=u=v=1$ and  $i=j=2$ in Theorem \ref{MainThm3}, we obtain
\begin{equation*}   \varphi^{2}  \left( q \right)  \varphi  \left( {q}^{2}
 \right) =\varphi  \left( {q}^{2} \right)   \varphi^{2}  \left( {q}^
{4} \right)  +4\,q\psi \left( {q}^{4} \right)   \psi^{2}
 \left( {q}^{2} \right) +4\,{q}^{2}\varphi  \left( {q}^{2
} \right)  \psi^{2} \left( {q}^{8} \right),
\end{equation*}
from which, we deduce
\begin{equation*}
  r(1,1,2;N) =  r(2,4,4;N) +  4 \, T(2,2,4;N-1)  +  4\,  rT(2,8,8;N-2).
\end{equation*}
Thus, using the above identity,  we obtain the desired result.
\end{proof}
\begin{remark}  Theorem \ref{Athm3}, can be written in the following form:
 For each non-negative integer $N$, we have
 \begin{align*}
  r(1,1,2;2N) =&r(1,2,2;N)+4\,rT(1,4,4;N-1),\\
  r(1,1,2;2N+1) =& 4 \, T(1,1,2;N).
 \end{align*}

\end{remark}
Using the above theorem and the fact that, any non-negative integer can be written in the form $t_l+t_m+2\,t_n$ which is due to Liouville, we deduce the following result, which is due to L. Panaitopol \cite{Panaitopol}:
 \begin{corollary}
Any positive odd integer can be written as $l^2+m^2+2\,n^2$ with $l,m,n \in \mathbb{Z}$.
\end{corollary}
\begin{theorem}\label{Athm4}
For each non-negative integer $N$, we have
\begin{align}
 T(1,1,2;4\,N)=&Rt(1,2,1;N),  \label{Athm6.1}\\
 T(1,1,2;4\,N+1)=&2\,rT(2,1,2;N),\label{Athm6.2}\\
   T(1,1,2;4\,N+2)=&2\,rT(1,1,4;N),\label{Athm6.3}\\
      T(1,1,2;4\,N+3)=&4\,T(1,2,4;N),\label{Athm6.4}\\
  r(1,2,4;2\,N) =& r(1,2,2;N),\label{Athm6.5}\\
  r(1,2,4;2\,N+1) =& 2 \, T(1,1,2;N)\label{Athm6.6}.
\end{align}
 \end{theorem}

\begin{proof}
Setting $k=2$, $r=1$, $\epsilon_1=\epsilon_2=\epsilon_3=1$, $g=u=3$ $h=v=1$, $i=6$ and $j=2$ in Theorem \ref{MainThm3}, we obtain
\begin{align}\label{Athm5.1}
     \psi^2(q)\psi\left(q^2 \right)=& \varphi\left(q^4 \right) \psi\left(q^4 \right) \varphi\left(q^8 \right)+2\,q \varphi\left(q^8 \right) \psi\left(q^4 \right)\psi\left(q^8 \right)\nonumber \\&+2\,q^2 \varphi\left(q^4 \right) \psi\left(q^4 \right)\psi\left(q^{16} \right)+4\,q^3 \psi\left(q^4 \right) \psi\left(q^8 \right)\psi\left(q^{16} \right).
   \end{align}
   Extracting the terms in the above identity involving $q^{4\,N}$, we find that
   \begin{equation*}
     \sum_{N=0}^{\infty}T(1,1,2;4\,N)q^{4\,N}=\varphi\left(q^4 \right) \psi\left(q^4 \right) \varphi\left(q^8 \right),
   \end{equation*}
   replacing $q^4$ by $q$, and then equating the coefficients of $q^N$ on both sides of the above equality, we obtain  \eqref{Athm6.1}.\\
    Extracting the terms in \eqref{Athm5.1}  involving $q^{4\,N+1}$,  then dividing the resulting identity by $q$ and replacing $q^4$ by $q$, we have
      \begin{equation*}
     \sum_{N=0}^{\infty}T(1,1,2;4\,N+1)q^{N}=2\, \varphi\left(q^2 \right) \psi\left(q \right)\psi\left(q^2 \right)=2\sum_{N=0}^{\infty}rT(2,1,2;N)q^{N}.
   \end{equation*}
     Equating the coefficients of $q^N$ on both sides, we obtain \eqref{Athm6.2}. In a similar way, we deduce  \eqref{Athm6.3} and \eqref{Athm6.4}.
    Using \eqref{Athm5.1} and \eqref{varphi=}, we have
\begin{align*}
 \psi^2(q)\psi\left(q^2 \right)=& \varphi\left(q^4 \right) \psi\left(q^4 \right) \left\{\varphi\left(q^8 \right) +2\,q^2  \psi\left(q^{16} \right) \right\}  \\
&+2\,q \psi\left(q^8 \right) \psi\left(q^4 \right) \left\{\varphi\left(q^8 \right) +2\,q^2  \psi\left(q^{16} \right) \right\} \\
=&\varphi\left(q^2 \right) \psi\left(q^4 \right) \left\{\varphi\left(q^4 \right) +2\,q^2  \psi\left(q^{8} \right) \right\} \\
=&\varphi\left(q \right)\varphi\left(q^2 \right) \psi\left(q^4 \right),
\end{align*}
Now, changing $q$ to $q^2$ in the above identity, then  multiplying both sides of the resulting identity by $2\,q$ and applying \eqref{varphi=}, we find that
\begin{equation*}
  \varphi\left(q \right)  \varphi\left(q^2 \right)  \varphi\left(q^4 \right) =  \varphi\left(q^2 \right)  \varphi^2 \left(q^4 \right)   + 2\,q  \psi^2\left(q^2 \right) \psi\left(q^4 \right),
\end{equation*}
from which we obtain
\begin{equation*}
  r(1,2,4;N) =r(2,4,4;N)+2\, T(2,2,4;N-1),
\end{equation*}
which implies  \eqref{Athm6.5} and \eqref{Athm6.6}.
\end{proof}
\begin{theorem}\label{Athm7} For any non-negative integer $N$, we have
\begin{align}\label{Athm7.1}
 G(1,1,2;2N)=&Rg(3,6,2;N)+2\,rtp(3,12,4;N-1),\\
 G(1,1,2;2N+1)=&2\,Tp(3,6,1;N). \label{Athm7.2}
\end{align}
\end{theorem}
\begin{proof} Setting $k=2$, $r=1$, $\epsilon_1=\epsilon_2=\epsilon_3=1$, $g=u=5$ $h=v=1$, $i=10$ and $j=2$ in Theorem \ref{MainThm3} and then using \eqref{X(q)=}, we obtain
\begin{equation}\label{Athm7.3}
  Y^2(q)Y(q^2)=\varphi(q^6)\varphi(q^{12})Y(q^4)+2\,q \psi(q^6)\psi(q^{12})X(q^2)+2\,q^2 \varphi(q^6)\psi(q^{24})X(q^8),
\end{equation}
Now, collecting the terms in \eqref{Athm7.3} that are  involving $q^{2\,N}$, and then replacing $q^2$ by $q$, obtain
\begin{align*}
  \sum_{N=0}^{\infty}G(1,1,2;2\,N)q^{N}&=\varphi(q^3)\varphi(q^{6})Y(q^2)+2\,q \varphi(q^3)\psi(q^{12})X(q^4)\\
  &= \sum_{N=0}^{\infty}  [ Rg(3,6,2;N)+2\,rtp(3,12,4;N-1) ]q^{N}.
\end{align*}
Equating the coefficients of $q^N$ on both sides of the above equality, we obtain \eqref{Athm7.1}. In a similar way, we deduce \eqref{Athm7.2}.
\end{proof}
\begin{theorem}\label{Athm9} For any non-negative integer $N$, we have
\begin{align*}
pG(4,1,1;2\,N)=&rtp(3,3,1;N),\\
pG(4,1,1;4\,N+1)=&2\, rtp(3,3,2;N)+4Tg(3,6,1;N-1),\\
pG(4,1,1;4\,N+3)=&0.
\end{align*}
\end{theorem}
\begin{proof}
Using Theorem \ref{MainThm3}, with $k=2$, $r=1$, $\epsilon_1=\epsilon_2=\epsilon_3=1$, $g=u=5$ $h=v=1$, $i=8$ and $j=4$,
and then setting $a=q^2$ and $b=q^{4}$ in \eqref{P2.10},    we find that
\begin{equation*}
X(q^4)Y^2(q)= \varphi(q^{6})\psi(q^{6})X(q^2)+2\,q \varphi(q^{12})\psi(q^{12})X(q^8)+4\,q^5 \psi(q^{12})\psi(q^{24})Y(q^4).
\end{equation*}
Using the above identity, we deduce
 \begin{equation*}
 pG(4,1,1;N)=rtp(6,6,2;N)+2\, rtp(12,12,8;N-1)+4Tg(12,24,4;N-5),
 \end{equation*}
 which leads us to the desired results.
\end{proof}
Using Theorem \ref{Athm9}, we deduce the following result:
\begin{corollary}
Each non-negative integer of the form $4\,k+3, (k \in \mathbb{N}_0)$ cannot be represented in the form $4\,p_l+g_m+g_n$, $(l, m,n \in \mathbb{Z})$.
\end{corollary}
\begin{theorem}\label{Athm8} For any non-negative integer $N$, we have
\begin{align*}
tG(12,1,1;2\,N)&=rtg(3,6,1;N),\\
tG(12,1,1;4\,N+1)&=2\,tpg(3,2,1;N),\\
tG(12,1,1;4\,N+3)&=0.
\end{align*}
\end{theorem}
\begin{proof}
Using Theorem \ref{MainThm3}, with $k=2$, $r=1$, $\epsilon_1=\epsilon_2=\epsilon_3=1$, $g=u=5$ $h=v=1$, $i=12$ and $j=0$,
and then applying \eqref{2.10P} with $a=q^2$ and $b=q^{10}$, we find that
\begin{equation*}
  \psi(q^{12})Y^2(q)=\varphi(q^{6})\psi(q^{12})Y(q^2)+2\,q\psi(q^{12})X(q^8)Y(q^4).
\end{equation*}
Using the above identity, we deduce the desired results.
\end{proof}
The third equality in Theorem \ref{Athm8} implies the following result:
\begin{corollary}
Each non-negative integer of the form $4\,k+3, (k \in \mathbb{N}_0)$ cannot be represented in the form $12\,t_l+g_m+g_n$, $(l, m,n \in \mathbb{Z})$.
\end{corollary}
\begin{theorem}\label{Athm10} For any non-negative integer $N$, we have
\begin{align*}
Rg(3,3,2;2\,N)&=rP(3,4,4;N)+rG(3,2,2;N-1),\\
Rg(3,3,2;2\,N+1)&=4\,Tg(6,6,1;N-1).
\end{align*}
\end{theorem}
\begin{proof}Setting $k=2$, $r=1$, $\epsilon_1=\epsilon_2=\epsilon_3=1$, $g=u=h=v=3$, $i=10$ and $j=2$ in Theorem \ref{MainThm3} and then applying \eqref{2.10P} with $a=q^2$ and $b=q^{10}$, in the resulting identity, we obtain
\begin{equation}\label{Athm10.1}
  \varphi^2(q^3)Y(q^2)=\varphi(q^6)X^2(q^8)+q^2\varphi(q^6)Y^2(q^4)+4\,q^3\psi^2(q^{12})Y(q^2).
\end{equation}
Using \eqref{Athm10.1}, we deduce the required results.
\end{proof}
\begin{theorem}\label{Athm11} For any non-negative integer $N$, we have
\begin{align*}
Rp(3,3,4;2\,N)&=rP(3,1,1;N)-2rtg(3,6,1;N-1),\\
Rp(3,3,2;4\,N+3)&=4\,tpg(3,2,1;N),\\
Rp(3,3,2;4\,N+1)&=0.
\end{align*}
\end{theorem}
\begin{proof}
Setting $k=2$, $r=1$, $\epsilon_1=\epsilon_2=\epsilon_3=1$, $g=u=h=v=3$, $i=8$ and $j=4$ in Theorem \ref{MainThm3}, then applying \eqref{P2.10} with $a=q^2$ and $b=q^{4}$, by squaring both sides, and finally using \eqref{2.10P} with $a=q^2$ and $b=q^{10}$, in the resulting identity, we obtain
\begin{equation*}
  \varphi^2(q^3)X(q^4)=\varphi(q^6)X^2(q^2)-2\,q^2\varphi(q^6)\psi(q^{12})Y(q^2)+4\,q^3\psi(q^{12})X(q^8)Y(q^4),
\end{equation*}
from which we obtain the desired results.
\end{proof}
The third equality in Theorem \ref{Athm11} implies the following result:
\begin{corollary}
Each non-negative integer of the form $4\,k+1, (k \in \mathbb{N}_0)$ cannot be represented in the form $3\,l^2+3\,m^2+4\,p_n$, $(l, m,n \in \mathbb{Z})$.
\end{corollary}

\begin{theorem}\label{Athm12} For any non-negative integer $N$, we have
\begin{align*}
  pG(2,1,1;6\,N)&=rP(1,1,2;N),\\
  pG(2,1,1;6\,N+1)&=2\,rtp(3,2,1;N), \\
   pG(2,1,1;6\,N+2)&=2\,rtp(1,3,2;N),\\
   pG(2,1,1;6\,N+3)&=2\,tpg(2,1,1;N),\\
   pG(2,1,1;6\,N+4)&=2\,rtp(1,6,1;N),\\
   pG(2,1,1;6\,N+5)&=4\,Tg(2,3,1;N).
\end{align*}
\end{theorem}
\begin{proof} Setting $k=3,\quad r=2, \quad \epsilon_1=\epsilon_2=\epsilon_3=1, \quad g=u=5, \quad h=v=1, \quad i=4$ and $j=2$ in Theorem \ref{MainThm3}, we obtain
\begin{align}
 \nonumber X(q^2)&Y^2(q)=\varphi(q^6)X(q^6)X(q^{12})+2\,q\varphi(q^{18})\psi(q^{12})X(q^6)+2\,q^2\varphi(q^{6})\psi(q^{18})X(q^{12})\\&+2\,q^3\psi(q^{12})X(q^6)Y(q^6)
+2\,q^4\varphi(q^{6})\psi(q^{36})X(q^{6})+4\,q^5\psi(q^{18})  \psi(q^{12}) Y(q^6). \label{Athm12.11}
\end{align}
Using  identity \eqref{Athm12.11}, we deduce the asserted  results.
\end{proof}
The proofs of the following results follow in similar ways as the previous theorems, using Theorem \ref{MainThm3} with help of some identities in Section \ref{Preliminary}, so we omit the details.
  \begin{theorem}\label{AAthm3}
For any non-negative integer $N$, we have
\begin{align*}
2\,T(1,1,4;N)&=rpg(1,2,2;N)+Rt(1,6,6;N),\\
2\,T(1,1,4;2\,N)&=rpg(2,1,1;N)+Rt(2,3,3;N),\\
T(1,1,4;2\,N+1)&=tpg(4,1,1;N)+rT(3,3,4;N),\\
4\,T(1,1,8;N)&=r(1,4,4;4N+5)-r(1,4,16;4N+5),\\
T(1,1,8;2\,N)&=T(1,2,2;N),\\
T(1,1,8;2\,N+1)&=2T(1,4,4;N),\\
r(1,1,1;N)&=r(1,1,4;N)+r(1,4,4;4\,N)-r(1,4,16;4\,N),\\
r(1,4,4;4\,N+2)&=r(1,4,16;4\,N+2),\\
r(1,4,4;4\,N+3)&=r(1,4,16;4\,N+3),\\
T(2,3,3;4N)&=Rt(3,3,2;N),\\
T(2,3,3;4N+1)&=2\,T(1,3,12;2\,N-1),\\
T(2,3,3;4N+2)&=Rt(1,3,6;N),\\
T(2,3,3;4N+3)&=2\,T(1,3,12;2\,N),\\
T(2,7,7;4\,N)&=Rt(7,14,4;N)+Rt(2,7,12;N-3),\\
T(2,7,7;4\,N+2)&=rT(7,1,7;N),\\
    T(2,7,7;8\,N+3)&=2\,rT(1,7,14;N-2),\\
    T(2,7,7;8\,N+7)&=2\,rT(7,2,7;N),\\
    T(2,7,7;4\,N+1)&=2\,T(1,7,14;N-2),\\
T(2,5,5;4\,N)&=Rt(5,30,3;N)+rpg(5,1,10;N-3),\\
&=rP(5,4,5;N)+2\,rT(5,12,15;N-3),\\
T(2,5,5;4\,N+1)&=2tP(10,4,5;N-1)+2\,T(10,12,15;N-4),\\
&=2rT(30,10,3;N-1)+2tpg(10,1,10;N-4),\\
T(2,5,5;4\,N+2)&=rP(5,1,20;N)+2rT(5,3,60;N-7),\\
&=rpg(5,5,2;N)+Rt(5,6,15;N-1),\\
T(2,5,5;4\,N+3)&=2tpg(10,20,1;N-1)+4T(3,10,6;N-8),\\
T(2,5,5;20\,N+11)&=2tpg(10,5,2;5N+1)+2rT(12,2,3;N),\\
T(2,5,5;20\,N+3)&=2tpg(10,5,2;5N-1),\\
T(2,5,5;20\,N+7)&=2tpg(10,5,2;5N),\\
T(2,5,5;20\,N+15)&=2tpg(10,5,2;5N+2),\\
T(2,5,5;20\,N+19)&=2tpg(10,5,2;5N+3),\\
T(2,15,15;4\,N)&=rT(15,3,5;N),\\
T(2,15,15;4\,N+2)&=Rt(10,15,12;N)+Rt(6,15,20;N-1),\\
T(2,15,15;4\,N+3)&=2T(3,5,30;N-3),\\
T(2,15,15;8\,N+1)&=2rT(5,6,15;N-2),\\
T(2,15,15;8\,N+5)&=2rT(3,10,15;N-2).
\end{align*}
\end{theorem}
\begin{theorem}For any non-negative integer $N$, we have
\begin{align*}
rtg(3,12,1;4\,N)&=Pg(1,2,1;N),\\
rtg(3,12,1;2\,N+1)&=tP(6,1,1;N),\\
rtg(3,12,1;4\,N+2)&=0,\\
tpg(3,1,1;2N)&=rpg(3,1,1;N),\\
tpg(3,1,1;2N+1)&=2tP(6,2,1;N),\\
Pg(2,2,1;3N)&=Rp(2,3,4;N)+2rtg(3,4,2;N-1),\\
Pg(2,2,1;3N+1)&=rpg(2,4,1;N)+2tG(4,2,1;N-1),\\
Pg(2,2,1;3N+2)&=2rtg(6,4,1;N)+2rtg(2,12,1;N-1),\\
Pg(1,1,1;2N)&=rtp(3,3,1;2N)+4\,Tp(3,6,2;N-1),\\
Pg(1,1,1;2N+1)&=rtp(3,3,1;2N+1)+2\,rtg(3,3,1;N),\\
Tg(3,3,1;2N)&=rtg(3,6,1;2N)+tP(3,2,2;N),\\
Tg(3,3,1;2N+1)&=rtg(3,6,1;2N+1)+tG(3,1,1;N).
\end{align*}
\end{theorem}

\begin{theorem}\label{AAthm71}
For any non-negative integer $N$, we have
\begin{align*}
Rt(1,1,2;2\,N)&=rP(1,1,2;N)+2\,rT(1,3,6;N),\\
Tg(3,3,1;2N+1)&=rtg(3,6,1;2N+1)+tG(3,1,1;N),\\
Rt(1,1,2;2\,N)&=rP(1,1,2;N)+2\,rT(1,3,6;N),\\
Rt(1,1,2;2\,N+1)&=2tpg(2,1,1;N)+2rT(3,2,3;N),\\
rG(6,1,1;2\,N)&=rP(3,4,4;N)+rG(3,2,2;N-1),\\
rG(6,1,1;2\,N+1)&=2tP(6,1,1;N)-4Tg(6,6,1;N-1),\\
tP(6,1,1;N)&=rpg(3,4,2;N)+2Tg(6,6,1;N-1),\\
Rg(3,6,1;2N)&=P(1,1,2;N)-2tpg(6,2,1;N-1),\\
Rg(3,6,1;2N+1)&=Pg(4,4,1;N)+G(1,2,2;N-1),\\
rpg(3,4,1;4N)&=rP(3,1,2;N)+2tpg(6,1,1;N-1),\\
rpg(3,4,1;2N+1)&=tpg(3,1,1;N),\\
rpg(3,4,1;4N+2)&=0,\\
tpg(3,2,1;N)=&rtp(12,6,1;N)+2Tp(6,24,1;N-3),\\
=&rtp(3,6,1;N),\\
tG(4,1,1;4N+1)&=Tp(1,3,1;N)+rtp(2,3,4;N)+2Tg(3,4,2;N-1),\\
tG(4,1,1;4N+3)&=0,\\
2rtg(3,4,1;4\,N)&=rP(2,1,4;N)+tP(1,1,1;N)+2tpg(4,1,2;N-1).
\end{align*}
\end{theorem}
\begin{corollary} For $k\in\mathbb{N}_{0}$, we have the following:      
\begin{enumerate}
\item Each non-negative integer of the form $4\, k +2$ cannot be
represented in the form $3 l^2 + 12t_m +  g_n$, ($l, m, n \in \mathbb{Z}$).
    \item Each non-negative integer of the form $4\, k +2$ cannot be
represented in the form $3 l^2 + 4p_m +  g_n$, ($l, m, n\in \mathbb{Z}$).
\item Each non-negative integer of the form $4\, k +3$ cannot be
represented in the form $4 t_l + g_m +  g_n$, ($l, m, n \in \mathbb{Z}$).
\item Each non-negative integer of the form $4\, k +2$ cannot be
represented in the form $3l^2 +4 t_m +  g_n$, ($l, m, n \in \mathbb{Z}$).
\item Each non-negative integer of the form $4\, k +2$ cannot be
represented in the form $9l^2 +4 p_m + 3 g_n$, ($l, m, n \in \mathbb{Z}$).
\item Each non-negative integer of the form $4\, k +1$ cannot be
represented in the form $3l^2 +3 m^2 + 4 t_n$, ($l, m, n \in \mathbb{Z}$).
\end{enumerate}
\end{corollary}

\begin{theorem}\label{AAthm18}
For any non-negative integer $N$, we have
\begin{align*}
rtg(3,4,1;4\,N+2)&=0,\\
rpg(9,4,3;4N)&=rP(18,3,4;N)+tP(9,1,3;N-1)+2tpg(36,3,2;N-5),\\
rpg(9,4,3;4N+2)&=0,\\
Rt(3,3,4;4N+3)&=2T(1,3,3;N)+2rT(6,3,4;N)+2rT(2,3,6;N-1),\\
Rt(3,3,4;4N+1)&=0,\\
rtp(3,6,1;N)&=tP(3,2,8;N)+tpg(3,2,4;N-1),\\
tP(3,1,2;N)&=rpg(6,4,1;N)+2rG(12,1,2:N-2),\\
2Tp(3,6,1;N)&=tpg(6,2,1;N)+pG(4,1,2;N),\\
P(1,1,2;N)&=Rp(3,6,4;N)+2Tp(3,6,1;N-1)+2rtg(3,12,2;N-2),\\
r(2,3,3;2N)&=r(3,4,12;N)+2T(2,3,3;N-1)+4rT(3,8,24;N-4),\\
G(2,3,3;2N)&=Rg(9,36,4;N)+rtp(9,18,2;N-1)+2rtp(9,72,8;N-8),\\
rG(2,1,1;2N)&=rtp(3,2,2;N)+Rg(3,4,4;N-1)+2rtp(3,8,8;N-1),\\
Rg(2,3,1;2N+1)&=rG(4,1,4;N-1)+tpg(2,2,1;N-1)+2tpg(8,8,1;N-2).
\end{align*}
\end{theorem}
\bibliographystyle{siam}
\bibliography{myref}

\end{document}